\documentclass[pagesize,pdftex]{scrartcl}
\usepackage{latexsym} 
\usepackage{amsmath}
\usepackage{amsthm}
\usepackage{amsfonts}
\usepackage{amssymb}
\usepackage{amsxtra}
\usepackage{amscd}
\usepackage{ifthen}
\usepackage{graphicx}
\usepackage{enumerate}
\usepackage{mathrsfs}
\usepackage[pagebackref=true]{hyperref}

\hypersetup{
pdfauthor={Mads Kyed, Jonas Sauer},
pdftitle={On Time-Periodic Solutions to Parabolic Boundary Value Problems of Agmon-Douglis-Nirenberg Type},
breaklinks=true,
colorlinks=true,
linkcolor=blue,
citecolor=blue,
urlcolor=blue,
filecolor=blue,
}

\numberwithin{equation}{section} 
\setkomafont{title}{\normalfont}

\newcommand{\myint}[3]{\int_{{#1}}{#2}\,\text{d}{#3}}

\newcommand{\kt}{{k}}

\newcommand{\teta}{\widetilde{\eta}}
\newcommand{\txip}{\widetilde{\xi}'}
\newcommand{\txi}{\widetilde{\xi}}
\newcommand{\inverse}{{-1}}
\newcommand{\pex}{\parnorm{\eta}{\xi}}
\newcommand{\pexp}{\parnorm{\eta}{\xi'}}
\newcommand{\pkx}{\parnorm{k}{\xi}}
\newcommand{\pkxp}{\parnorm{k}{\xi'}}

\newcommand{\multrestriction}[2]{#1_{|#2}}
\newcommand{\funcrestriction}[2]{#1_{|#2}}
\newcommand{\Dtracekonst}{\iota}
\newcommand{\Btracekonst}{\kappa}
\newcommand{\hgrp}{{\torus\times\R^{n-1}}}
\newcommand{\xin}{z}
\newcommand{\half}{\frac{1}{2}}
\newcommand{\cutoff}{\chi}

\newcommand{\uf}{u}
\newcommand{\vf}{v}
\newcommand{\wf}{w}

\newcommand{\f}{f}
\newcommand{\g}{g}

\newcommand{\per}{T}
\newcommand{\iper}{\frac{1}{\per}}

\newcommand{\perf}{\frac{2\pi}{\per}}

\newcommand{\dualrho}{\hat{\rho}}
\newcommand{\charplus}{Y_+}

\newcommand{\magicvector}{L}
\newcommand{\minusnullset}{\setminus\set{(0,0)}}
\newcommand{\xip}{{\xi'}}
\newcommand{\xitan}{{\zeta}}
\newcommand{\Mmultiplier}{M}

\newenvironment{pdeq}{ \left\{ \begin{aligned}}{\end{aligned}\right.}

\newcommand{\np}[1]{(#1)}
\newcommand{\nb}[1]{[#1]}
\newcommand{\bp}[1]{\big(#1\big)}
\newcommand{\bb}[1]{\big[#1\big]}
\newcommand{\Bp}[1]{\bigg(#1\bigg)}

\newcommand{\calp}{{\mathcal P}}

\newcommand{\calr}{{\mathcal R}}

\newcommand{\R}{\mathbb{R}}
\newcommand{\Z}{\mathbb{Z}}
\newcommand{\C}{\mathbb{C}}
\newcommand{\CNumbers}{\mathbb{C}}
\newcommand{\N}{\mathbb{N}}
\DeclareMathOperator{\Ext}{Ext}

\DeclareMathOperator{\impart}{Im}
\DeclareMathOperator{\e}{e}

\DeclareMathOperator{\id}{\textsf{id}}

\DeclareMathOperator{\cof}{cof}
\DeclareMathOperator{\supp}{supp}

\DeclareMathOperator{\trace}{Tr}

\newcommand{\embeds}{\hookrightarrow}

\newcommand{\elop}{A^H}
\newcommand{\bdop}{B^H}
\newcommand{\elopfull}{A}
\newcommand{\bdopfull}{B}
\newcommand{\parop}{\mathsf{A}}
\newcommand{\paropinv}{\parop^{-1}}
\newcommand{\paropp}{\parop_+}
\newcommand{\paropm}{\parop_-}
\newcommand{\paroppm}{\parop_{\pm}}
\newcommand{\paroppinv}{\paropp^{-1}}
\newcommand{\paropminv}{\paropm^{-1}}
\newcommand{\paroppminv}{\paroppm^{-1}}

\newcommand{\Fparop}{\frak{M}}
\newcommand{\Fparopcoef}{c}
\newcommand{\Fparopper}{\mathsf{M}}
\newcommand{\mult}{\frak{m}}
\newcommand{\tmult}{\frak{\widetilde{m}}}
\newcommand{\Mult}{\frak{M}}
\newcommand{\multper}{\mathsf{m}}
\newcommand{\Ra}{\Rightarrow}

\newcommand{\ra}{\rightarrow}

\newcommand{\set}[1]{\ensuremath{\{#1\}}}

\newcommand{\setc}[2]{\ensuremath{\{#1\ \lvert\ #2\}}}
\newcommand{\setcl}[2]{\ensuremath{\bigl\{#1\ \lvert\ #2\bigr\}}}

\newcommand{\closure}[2]{\overline{#1}^{#2}}
\newcommand{\seqN}[1]{\ensuremath{\set{#1}_{n=1}^\infty}}

\newcommand{\proj}{\calp}

\newcommand{\projcompl}{\calp_\bot}

\newcommand{\quotientmap}{\pi}
\newcommand{\grp}{G}
\newcommand{\dualgrp}{\widehat{G}}

\newcommand{\torus}{{\mathbb T}}

\newcommand{\grpprime}{{H}}
\newcommand{\transpose}{\top}

\newcommand{\charmatrix}{F}
\newcommand{\invcharmatrix}{F^{-1}}
\newcommand{\halfgrp}{\torus\times\halfspace}

\newcommand{\halfgrpclosed}{\torus\times\overline{\halfspace}}
\newcommand{\lhalfgrpclosed}{\torus\times\overline{\lhalfspace}}

\newcommand{\wspace}{\R^n}
\newcommand{\halfspace}{\wspace_+}
\newcommand{\lhalfspace}{\wspace_-}

\newcommand{\dx}{{\mathrm d}x}

\newcommand{\ds}{{\mathrm d}s}
\newcommand{\dt}{{\mathrm d}t}
\newcommand{\darg}[1]{{\mathrm d}#1}

\newcommand{\dxi}{{\mathrm d}\xi}

\newcommand{\Hardy}{\mathscr{H}}
\newcommand{\SR}{\mathscr{S}}
\newcommand{\SRcompl}{\SR_\bot}

\newcommand{\TDR}{\mathscr{S^\prime}}
\newcommand{\TDRcompl}{\mathscr{S^\prime_\bot}}

\newcommand{\ft}[1]{\widehat{#1}}
\newcommand{\ift}[1]{{#1}^\vee}
\newcommand{\FT}{\mathscr{F}}
\newcommand{\iFT}{\mathscr{F}^{-1}}

\newcommand{\op}{\mathsf{op}\,}
\newcommand{\norm}[1]{\lVert#1\rVert}

\newcommand{\norml}[1]{\bigl\lVert#1\bigr\rVert}
\newcommand{\snorm}[1]{{\lvert #1 \rvert}}

\newcommand{\opnorm}[1]{{\lvert\kern-0.25ex\lvert\kern-0.25ex\lvert #1 \rvert\kern-0.25ex\rvert\kern-0.25ex\rvert}}
\newcommand{\parnorm}[2]{|#1,#2|}

\newcommand{\WSR}[2]{W^{#1,#2}} 

\newcommand{\HSR}[2]{H^{#1,#2}}

\newcommand{\WSRloc}[2]{W^{#1,#2}_{\mathrm{loc}}} 
\newcommand{\CR}[1]{C^{#1}}  
\newcommand{\LR}[1]{L^{#1}}
\newcommand{\LRcompl}[1]{L^{#1}_\bot}

\newcommand{\LRloc}[1]{L^{#1}_{\mathrm{loc}}} 
\newcommand{\CRi}{\CR \infty}
\newcommand{\CRci}{\CR \infty_0}

\newcommand{\tracespacecompl}[2]{T_\bot^{#1,#2}}
\newcommand{\VR}{V} 
\newcommand{\LRper}[1]{L_{\mathrm{per}}^{#1}}

\newcommand{\CRciper}{\CR{\infty}_{0,\mathrm{per}}}

\newcommand{\WSRper}[2]{W^{#1,#2}_{\mathrm{per}}}

\newcommand{\WSRcompl}[2]{W^{#1,#2}_{\bot}} 
\newcommand{\HSRcompl}[2]{H^{#1,#2}_{\bot}} 
 
\newcommand{\tin}{\text{in }}

\newcommand{\ton}{\text{on }}

\newcommand{\tand}{\text{and }}

\newcommand{\newCCtr}[2][d]{
\newcounter{#2}\setcounter{#2}{0}
\expandafter\xdef\csname kyedtheconst#2\endcsname{#1}
}
\newcommand{\Cc}[2][nolabel]{
\stepcounter{#2}
\expandafter\ensuremath{\csname kyedtheconst#2\endcsname_{\arabic{#2}}}
\ifthenelse{\equal{#1}{nolabel}}
{}
{\expandafter\xdef\csname kyedconst#1\endcsname
{\expandafter\ensuremath{\csname kyedtheconst#2\endcsname_{\arabic{#2}}}}}
}
\newcommand{\Ccn}[2][nolabel]{
\expandafter\ensuremath{\csname kyedtheconst#2\endcsname}
\ifthenelse{\equal{#1}{nolabel}}
{}
{\expandafter\xdef\csname kyedconst#1\endcsname
{\expandafter\ensuremath{\csname kyedtheconst#2\endcsname}}}
}
\newcommand{\CcSetCtr}[2]{
\setcounter{#1}{#2}
}
\newcommand{\Cclast}[1]{
\expandafter\ensuremath{\csname kyedtheconst#1\endcsname_{\arabic{#1}}}
}
\newcommand{\Ccnlast}[1]{
\Ccn{#1}
}
\newcommand{\Ccllast}[1]{
\addtocounter{#1}{-1}
\expandafter\ensuremath{\csname kyedtheconst#1\endcsname_{\arabic{#1}}}
\addtocounter{#1}{1}
}
\newcommand{\const}[1]{
\expandafter{\ifcsname kyedconst#1\endcsname
  \csname kyedconst#1\endcsname
\else
  \errmessage{Undefined Kyedconstant #1.}%
\fi}
}

\renewcommand{\frak}{\mathfrak}

\theoremstyle{plain}
\newtheorem{thm}{Theorem}[section]
\newtheorem{defn}[thm]{Definition}
\newtheorem{lem}[thm]{Lemma}
\newtheorem{prop}[thm]{Proposition}
\newtheorem{cor}[thm]{Corollary}
\theoremstyle{remark}
\newtheorem{rem}[thm]{Remark}

\begin{document}
\title{On Time-Periodic Solutions to Parabolic Boundary Value Problems of Agmon-Douglis-Nirenberg Type}

\author{
Mads Kyed\\ 
Fachbereich Mathematik\\
Technische Universit\"at Darmstadt\\
Schlossgartenstr. 7, 64289 Darmstadt, Germany\\
Email: \texttt{kyed@mathematik.tu-darmstadt.de}
\and
Jonas Sauer\\ 
Max-Planck-Institut f\"ur Mathematik in den Naturwissenschaften\\
Inselstr. 22, 04103 Leipzig, Germany\\
Email: \texttt{sauer@mis.mpg.de}
}

\date{\today}
\maketitle

\begin{abstract}
Time-periodic solutions to partial differential equations of parabolic type corresponding to an operator that is elliptic in the sense of Agmon-Douglis-Nirenberg are investigated. 
In the whole- and half-space case we construct an explicit formula for the solution and establish coercive $\LR{p}$ estimates. 
The estimates generalize a famous result of Agmon, Douglis and Nirenberg for elliptic problems to the time-periodic case.   
\end{abstract}


\newCCtr[c]{c}
\newCCtr[C]{C}
\newCCtr[M]{M}
\newCCtr[B]{B}
\newCCtr[\epsilon]{eps}
\CcSetCtr{eps}{-1}

\section{Introduction}

We investigate time-periodic solutions to parabolic boundary value problems 
\begin{align}\label{intro_maineq}
\begin{pdeq}
\partial_t\uf+\elopfull\uf&=\f &&\tin\R\times\Omega,\\
\bdopfull_j\uf&=\g_j && \ton\R\times\partial\Omega,
\end{pdeq}
\end{align}
where $\elopfull$ is an elliptic operator of order $2m$ and $\bdopfull_1,\ldots,\bdopfull_m$ satisfy an appropriate complementing boundary condition.
The domain $\Omega$ is either the whole-space, the half-space or a bounded domain, and $\R$ denotes the time-axis. 
The solutions $\uf(t,x)$ correspond to time-periodic data $\f(t,x)$ and $\g_j(t,x)$ of the same (fixed) period $\per>0$. 
Using the simple projections
\begin{align*}
\proj\uf = \iper\int_0^\per \uf(t,x)\,\dt,\quad \projcompl:=\id-\proj,
\end{align*}
we decompose \eqref{intro_maineq} into an elliptic problem
\begin{align}\label{intro_projeq}
\begin{pdeq}
\elopfull\proj\uf&=\proj\f &&\tin\Omega,\\
\bdopfull_j\proj\uf&=\proj\g_j && \ton\partial\Omega,
\end{pdeq}
\end{align}
and a \emph{purely oscillatory} problem
\begin{align}\label{intro_projcompleq}
\begin{pdeq}
\partial_t\projcompl\uf+\elopfull\projcompl\uf&=\projcompl\f &&\tin\R\times\Omega,\\
\bdopfull_j\projcompl\uf&=\projcompl\g_j && \ton\R\times\partial\Omega.
\end{pdeq}
\end{align}
The problem \eqref{intro_projeq} is elliptic in the sense of Agmon-Douglis-Nirenberg, for which a comprehensive $\LR{p}$ theory was established in \cite{ADN1}.
In this article, we develop a complementary theory for the purely oscillatory problem \eqref{intro_projcompleq}. 
Employing ideas going back to \textsc{Peetre} \cite{Peetre61} and \textsc{Arkeryd} \cite{Arkeryd67}, we are able to establish an explicit formula for the solution to \eqref{intro_projcompleq} when the domain is either
the whole- or the half-space. 
We shall then introduce a technique based on tools from abstract harmonic analysis to show coercive $\LR{p}$ estimates. As a consequence, we obtain 
a time-periodic version of the celebrated theorem of  
\textsc{Agmon}, \textsc{Douglis} and \textsc{Nirenberg} \cite{ADN1}. 

The decomposition \eqref{intro_projeq}--\eqref{intro_projcompleq} is essential as the two problems
have substantially different properties.
In particular, we shall show in the whole- and half-space case that the principle part of the linear operator in the purely oscillatory problem 
\eqref{intro_projcompleq} is a 
homeomorphism in a canonical setting of time-periodic Lebesgue-Sobolev spaces. This is especially remarkable since the elliptic problem \eqref{intro_projeq} clearly does not satisfy this property. 
Another truly remarkable characteristic of \eqref{intro_projcompleq} is that the $\LR{p}$ theory we shall develop for this problem leads directly to a similar $\LR{p}$ theory, sometimes referred to as maximal regularity, for the parabolic initial-value problem associated to \eqref{intro_maineq}.

We consider general differential operators 
\begin{align}\label{DefOfDiffOprfull}
\elopfull(x,D):=\sum_{|\alpha|\leq 2m}a_\alpha(x) D^\alpha,\quad \bdopfull_j(x,D):=\sum_{|\alpha|\leq m_j}b_{j,\alpha}(x) D^\alpha\quad (j=1,\ldots,m) 
\end{align}
with complex coefficients 
$a_\alpha:\Omega\to\CNumbers$ and
$b_{j,\alpha}:\partial\Omega\to\CNumbers$. Here, $\alpha\in\N^{n}$ is a multi-index and 
$D^\alpha:=i^{|\alpha|}\partial_{x_1}^{\alpha_1}\ldots\partial_{x_n}^{\alpha_n}$.
We denote the principle part of the operators by
\begin{align}
\elop(x,D):=\sum_{|\alpha|= 2m}a_\alpha(x) D^\alpha,\quad \bdop_j(x,D):=\sum_{|\alpha|= m_j}b_{j,\alpha}(x) D^\alpha. 
\end{align}
We shall assume that $\elop$ is elliptic in the following classical sense:

\begin{defn}[Properly Elliptic]\label{TPADN_ProperlyElliptic}
The operator $\elop$ is said to be \emph{properly elliptic} if
for all $x\in\Omega$ and all $\xi\in\R^n\setminus\set{0}$ it holds $\elop(x,\xi)\neq 0$, and  
for all $x\in\Omega$ and all linearly independent vectors $\xitan,\xi\in\R^n$ the polynomial 
$P(\tau):= \elop(x,\xitan+\tau\xi)$ 
has $m$ roots in $\CNumbers$ with positive imaginary part, 
and 
$m$ roots in $\CNumbers$ with negative imaginary part. 
\end{defn}

Ellipticity, however, does not suffice to establish maximal $\LR{p}$ regularity for the time-periodic problem.
We thus recall Agmon's condition, also known as parameter ellipticity.

\begin{defn}[Agmon's Condition]\label{TPADN_AgmonCondA}
Let $\theta\in[-\pi,\pi]$. A properly elliptic operator $\elop$ is said to satisfy Agmon's condition on the ray $\e^{i\theta}$ if  
for all $x\in\Omega$ and all $\xi\in\R^n\setminus\set{0}$ it holds $\elop(x,\xi)\notin \setc{r\e^{i\theta}}{r\geq 0}$.
\end{defn}

If $\elop$ satisfies Agmon's condition on the ray $\e^{i\theta}$, then,
since the roots of a polynomial depend continuously on its coefficients, 
the polynomial 
$Q(\tau):=-r\e^{i\theta}+\elop(x,\xitan+\tau\xi)$
has $m$ roots $\tau_h^+(r\e^{i\theta},x,\xitan,\xi)\in\CNumbers$ with positive imaginary part,
and $m$ roots $\tau_h^-(r\e^{i\theta},x,\xitan,\xi)\in\CNumbers$ with negative imaginary part ($h=1,\ldots,m$). 
Consequently, the following assumption on the operator $(\elop,\bdop_1,\ldots,\bdop_m)$ is meaningful.

\begin{defn}[Agmon's Complementing Condition]\label{TPADN_AgmonCondAB}
Let $\theta\in[-\pi,\pi]$. If $\elop$ is a properly elliptic operator, then $(\elop,\bdop_1,\ldots,\bdop_m)$
is said to satisfy Agmon's complementing condition on the ray $\e^{i\theta}$ if: 
\begin{enumerate}[(i)]
\item $\elop$ satisfies Agmon's condition on the ray $\e^{i\theta}$.
\item For all $x\in\partial\Omega$, all pairs $\xitan,\xi\in\R^n$ with $\xitan$ tangent to $\partial\Omega$ and $\xi\in\R^n$ normal to $\partial\Omega$ at $x$, and all $r\geq 0$, let $\tau_h^+(r\e^{i\theta},x,\xitan,\xi)\in\CNumbers$ ($h=1,\ldots,m$) denote the $m$ roots of the polynomial 
${Q(\tau):=-r\e^{i\theta}+\elop(x,\xitan+\tau\xi)}$
with positive imaginary part. The polynomials
$P_j(\tau):=\bdop_j(x,\xitan+\tau\xi)$ ($j=1,\ldots,m$)  are linearly independent modulo the polynomial 
$\Pi_{h=1}^m \bp{\tau-\tau_h^+(r\e^{i\theta},x,\xitan,\xi)}$.
\end{enumerate}
\end{defn}

The property specified in Definition \ref{TPADN_AgmonCondAB} was first introduced by \textsc{Agmon} in \cite{Agmon62},
and later by \textsc{Agranovich} and \textsc{Vishik} in \cite{AgranovichVishik1964} as \textit{parameter ellipticity}. The condition was introduced in order to identify the additional requirements on the differential operators needed to extend the result of \textsc{Agmon}, \textsc{Douglis} and \textsc{Nirenberg} \cite{ADN1} from the 
elliptic case to the corresponding parabolic initial-value problem. 
The theorem of \textsc{Agmon}, \textsc{Douglis} and \textsc{Nirenberg} \cite{ADN1} requires $(\elop,\bdop_1,\ldots,\bdop_m)$ 
to satisfy Agmon's complementing condition only at the origin (\emph{not} on a full ray), in which case
$(\elop,\bdop_1,\ldots,\bdop_m)$ is said to be \textit{elliptic in the sense of Agmon-Douglis-Nirenberg}.
Analysis of the associated initial-value problem relies heavily on properties of the resolvent equation 
\begin{align}\label{intro_resolventeq}
\begin{pdeq}
\lambda\uf+\elop\uf&=\f &&\tin\Omega,\\
\bdop_j\uf&=0 && \ton\partial\Omega.
\end{pdeq}
\end{align}
It was shown by \textsc{Agmon} \cite{Agmon62} that a necessary and sufficient condition for the 
resolvent of $(\elop,\bdop_1,\ldots,\bdop_m)$ to lie in the negative complex half-plane (and thereby for the generation of an analytic semi-group) is that 
Agmon's complementing condition is satisfied for all rays with $\snorm{\theta}\geq\frac{\pi}{2}$. 
However, the step from 
analyticity of the semi-group to maximal $\LR{p}$ regularity for the parabolic initial-value problem
proved to be highly non-trivial. Although many articles were dedicated to this problem after the publication of \cite{Agmon62},
it was not until the celebrated work of \textsc{Dore} and \textsc{Venni} \cite{DoreVenni87} that a framework was developed with which maximal regularity could be established comprehensively from the assumption that Agmon's condition is satisfied for all rays with $\snorm{\theta}\geq\frac{\pi}{2}$. 
To apply \cite{DoreVenni87}, one has show that $(\elop,\bdop_1,\ldots,\bdop_m)$
admits bounded imaginary powers.
Later, it was shown that maximal regularity is in fact equivalent to $\calr$-boundedness of an appropriate resolvent family; see \cite{DenkHieberPruess_AMS2003}.  
Remarkably, our result for the time-periodic problem \eqref{intro_projcompleq} leads to a new and relatively short proof of 
maximal regularity for the parabolic initial-value problem \emph{without} the use of either bounded imaginary powers or the notion of $\calr$-boundedness; see Remark \ref{Intr_Rem} below. Under the assumption that $(\elop,\bdop_1,\ldots,\bdop_m)$ generates an analytic semi-group,
maximal regularity for the parabolic initial-value problem follows almost immediately as a corollary from our main theorem.
We emphasize that our main theorem of maximal regularity for the time-periodic problem does \emph{not} require 
the principle part of $(\elopfull,\bdopfull_1,\ldots,\bdopfull_m)$ to generate an analytic semi-group.
As a novelty of the present paper, and in contrast to the initial-value problem, we establish 
that maximal $\LR{p}$ regularity for the time-periodic problem requires Agmon's complementing condition to be satisfied only on the two rays with $\theta=\pm\frac{\pi}{2}$, that is, only on the imaginary axis. 

Our main theorem for the purely oscillatory problem \eqref{intro_projcompleq} concerns the half-space case and the question of existence of a unique solution satisfying  a coercive $\LR{p}$ estimate in 
the Sobolev space $\WSRper{1,2m}{p}(\R\times\R^n_+)$ of time-periodic functions 
on the time-space domain $\R\times\R^n_+$. We refer to Section \ref{pre} for definitions of the function spaces.
 
\begin{thm}\label{MainThm_HalfSpace}
Let $p\in (1,\infty)$, $\per>0$, $n\geq 2$. 
Assume that $\elop$ and $(\bdop_1,\ldots,\bdop_m)$ have constant coefficients. 
If $\elop$ is properly elliptic 
and $(\elop,\bdop_1,\ldots,\bdop_m)$ satisfies Agmon's complementing condition on the two rays $\e^{i\theta}$ with $\theta=\pm\frac{\pi}{2}$, 
then for all functions $\f\in\projcompl\LRper{p}(\R\times\halfspace)$
and $\g_j\in\projcompl\WSRper{\kappa_j,2m\kappa_j}{p}(\R\times\partial\halfspace)$ with $\kappa_j:=1-\frac{m_j}{2m}-\frac{1}{2mp}$ ($j=1,\ldots,m$)
there exists a unique solution $\uf\in\projcompl\WSRper{1,2m}{p}(\R\times\R^n_+)$
to 
\begin{align}\label{MainThm_HalfSpace_Eq}
\begin{pdeq}
\partial_t\uf+\elop\uf&=\f &&\tin\R\times\R^n_+,\\
\bdop_j\uf&=\g_j && \ton\R\times\partial\R^n_+.
\end{pdeq}
\end{align}
Moreover, 
\begin{align}\label{MainThm_HalfSpace_Est}
\begin{aligned}
\norm{\uf}_{\WSRper{1,2m}{p}(\R\times\R^n_+)}
\leq \Ccn{C} 
\Bp{\norm{\f}_{\LRper{p}(\R\times\R^n_+)}+
\sum_{j=1}^m \norm{\g_j}_{\WSRper{\kappa_j,2m\kappa_j}{p}(\R\times\partial\R^n_+)}}
\end{aligned}
\end{align}
with $\Ccnlast{C}=\Ccnlast{C}(p,\per,n)$.
\end{thm}

Our proof of Theorem \ref{MainThm_HalfSpace} contains two results that are interesting in their own right. Firstly, we 
establish a similar assertion in the whole-space case. Secondly, we provide an explicit formula for the solution; see \eqref{weak_sol_lem_repformula} below.  
Moreover, our proof is carried out fully in a setting of time-periodic functions and follows an argument adopted from 
the elliptic case. This is remarkable in view of the fact that analysis of time-periodic problems in existing literature 
typically is based on theory for the corresponding initial-value problem; see for example \cite{Lieberman99}. A novelty of our approach is the introduction of suitable tools from abstract harmonic analysis that
allow us to give a constructive proof and avoid completely the classical indirect characterizations of time-periodic solutions as
fixed points of a Poincar\'{e} map, that is, as 
special solutions to the corresponding initial-value problem. The circumvention of the initial-value problem 
also enables us to avoid having to assume Agmon's condition for all $\snorm{\theta}\geq\frac{\pi}{2}$ and instead carry out our
investigation under the the weaker condition that Agmon's condition is satisfied only for $\theta=\pm\frac{\pi}{2}$.

We shall briefly describe the main ideas behind the proof of Theorem \ref{MainThm_HalfSpace}.
We first consider the problem in the whole space $\R\times\R^n$ and replace the time axis $\R$ with the torus $\torus:=\R/\per\Z$
in order to reformulate the $\per$-time-periodic problem as a partial differential equation on the locally compact abelian group $\grp:=\torus\times\R^n$.
Utilizing the Fourier transform $\FT_\grp$ associated to $\grp$, we obtain an explicit representation formula for the time-periodic solution.
Since $\FT_\grp=\FT_\torus\circ\FT_{\R^n}$, this formula simply corresponds to a Fourier series expansion in time of the solution and subsequent Fourier
transform in space of all its Fourier coefficients. 
While it is relatively easy to obtain $\LR{p}$ estimates (in space) for each Fourier coefficient separately, 
it is highly non-trivial to deduce from these individual estimates an $\LR{p}$ estimate in space \emph{and} time via the corresponding Fourier series.  
Instead, we turn to the representation formula given in terms of $\FT_\grp$
and show that the corresponding Fourier multiplier defined on the dual group $\dualgrp$ is an $\LR{p}(\grp)$ multiplier. For this purpose, we use the so-called \emph{Transference Principle} for Fourier multipliers in a group setting, and obtain the necessary estimate in the whole-space case. 
In the half-space case, \textsc{Peetre} \cite{Peetre61} and \textsc{Arkeryd} \cite{Arkeryd67} utilized the Paley-Wiener Theorem 
in order to construct a representation formula for solutions to elliptic problems; see also \cite[Section 5.3]{Triebel_InterpolationTheory1978}. We adapt their ideas to our setting and establish $\LR{p}$ estimates from the ones already obtained in the whole-space case. 
  
Theorem \ref{MainThm_HalfSpace} can be reformulated as the assertion that the operator
\begin{multline*}
(\partial_t+\elop,\bdop_1,\ldots,\bdop_m):\\
\projcompl\WSRper{1,2m}{p}(\R\times\R^n_+)\ra\projcompl\LRper{p}(\R\times\halfspace)\times
\Pi_{j=1}^m \projcompl\WSRper{\kappa_j,2m\kappa_j}{p}(\R\times\partial\halfspace) 
\end{multline*}
is a homeomorphism. 
By a standard localization and perturbation argument,
a purely periodic version of the celebrated theorem of  
\textsc{Agmon}, \textsc{Douglis} and \textsc{Nirenberg} \cite{ADN1} in the general case of operators with variable coefficients and $\Omega$ being 
a sufficiently smooth domain follows.
In fact, combining the classical result \cite{ADN1} for the elliptic case with Theorem 
\ref{MainThm_HalfSpace}, we obtain the following time-periodic version of the Agmon-Douglis-Nirenberg Theorem: 

\begin{thm}[Time-Periodic ADN Theorem]\label{MainThm_TPADN}
Let $p\in (1,\infty)$, $\per>0$, $n\geq 2$ and $\Omega$ be a domain with a boundary that is uniformly $\CR{2m}$-smooth.
Assume $a_\alpha$ is bounded and uniformly continuous on $\overline{\Omega}$ for $\snorm{\alpha}=2m$, 
and $a_\alpha\in\LR{\infty}(\Omega)$ for $\snorm{\alpha}<2m$.  
Further assume $b_{j,\beta}\in\CR{2m-m_j}(\partial\Omega)$ with bounded and uniformly 
continuous derivatives up to the order {$2m-m_j$}. 
If $\elop$ is properly elliptic 
and $(\elop,\bdop_1,\ldots,\bdop_m)$ satisfies Agmon's complementing condition on the two rays $\e^{i\theta}$ with $\theta=\pm\frac{\pi}{2}$, 
then the estimate
\begin{align}\label{TPADNEst}
\begin{aligned}
&\norm{\uf}_{\WSRper{1,2m}{p}(\R\times\Omega)}\\
&\qquad\qquad\leq \Ccn{C} \Bp{\norm{\partial_t\uf + A\uf}_{\LRper{p}(\R\times\Omega)}+\sum_{j=1}^m 
\norm{\bdopfull_j\uf}_{\WSRper{\kappa_j,2m\kappa_j}{p}(\R\times\partial\Omega)}+ \norm{u}_{\LRper{p}(\R\times\Omega)}}
\end{aligned}
\end{align}
holds for all $\uf\in\WSRper{1,2m}{p}(\R\times\Omega)$, where $\kappa_j:=1-\frac{m_j}{2m}-\frac{1}{2mp}$ ($j=1,\ldots,m$). 
\end{thm}

Since time-independent functions are trivially also time-periodic, we have $\WSR{2m}{p}(\Omega)\subset\WSRper{1,2m}{p}(\R\times\Omega)$.
If estimate \eqref{TPADNEst} is restricted to functions in $\WSR{2m}{p}(\Omega)$, Theorem \ref{MainThm_TPADN} reduces to the classical theorem of Agmon-Douglis-Nirenberg \cite{ADN1}, 
which has played a fundamental role in the analysis of  
elliptic boundary value problems for more then half a century now. This classical theorem for scalar equations was extended to systems in \cite{ADN2}. We shall only treat scalar equations in the following, but we believe the method developed here could be extended to include systems.

Time-periodic problems of parabolic type have been investigated in numerous articles over the years, and it would be too far-reaching to 
list them all here. We mention only the article of \textsc{Liebermann} \cite{Lieberman99}, the recent article by \textsc{Geissert}, \textsc{Hieber} and \textsc{Nguyen} \cite{GeissertHieberNguyen16}, as well as the monographs \cite{HessBook91,VejvodaBook82}, and refer the reader to the references therein.
Finally, we mention the article \cite{KyedSauer17} by the present authors in which some of the ideas utilized in the following were introduced in a much simpler setting.

\begin{rem}\label{Intr_Rem}
The half-space case treated in Theorem \ref{MainThm_HalfSpace} 
is also pivotal in the $\LR{p}$ theory
for parabolic initial-value problems.
Denote by $\elop_\bdopfull$ the realization of the operator $\elop(D)$ in $\LR{p}(\R^n_+)$ with domain
\begin{align*}
 D(\elop_\bdopfull):=\setc{\uf\in \WSR{2m}{p}(\R^n_+)}{\bdop_j(D)\uf=0, \quad j=1,\ldots,m}.
\end{align*}
Maximal regularity for parabolic initial-value problems of Agmon-Douglis-Niren\-berg type
is based on an investigation of the initial-value problem  
\begin{align}\label{IVP}
\begin{pdeq}
 \partial_t\uf+\elop_\bdopfull\uf&=\f, &&t>0, \\
 \uf(0)&=0.
\end{pdeq}
\end{align}
Maximal regularity for \eqref{IVP} means that 
for each function $\f\in \LR{p}(0,\per;\LR{p}(\R^n_+))$ there is a unique solution $\uf\in \LR{p}(0,\per;D(\elop_\bdopfull))\cap \WSR{1}{p}(0,\per;\LR{p}(\R^n_+))$ which satisfies the estimate
\begin{align}\label{IVP_Est}
 \|\uf,\partial_t\uf, D^{2m}\uf\|_{\LR{p}(0,\per;\LR{p}(\R^n_+))}\leq c\|\f\|_{\LR{p}(0,\per;\LR{p}(\R^n_+))}.
\end{align}
We shall briefly sketch how to obtain maximal regularity for \eqref{IVP} from Theorem \ref{MainThm_HalfSpace}.
For this purpose, it is required that 
$-\elop_\bdopfull$ generates an analytic semi-group
$\set{\e^{-t\elop_\bdopfull}}_{t>0}$, which follows from resolvent estimates going back 
to \textsc{Agmon} \cite[Theorem 2.1]{Agmon62} derived under the assumption that 
$(\elop,\bdop_1,\ldots,\bdop_m)$ satisfies Agmon's complementing condition for all rays with $\snorm{\theta}\geq\frac{\pi}{2}$; see also 
\cite[Theorem 5.5]{TanabeBook}. We would like to point out that these resolvent estimates can also be established with the arguments in our proof of Theorem \ref{MainThm_HalfSpace}. One can periodically extend any 
$\f\in \LR{p}(0,\per;\LR{p}(\R^n_+))$ to a $\per$-periodic function 
$\f\in\LRper{p}(\R\times\halfspace)$. With $\uf$ denoting the solution from Theorem \ref{MainThm_HalfSpace} corresponding to $\projcompl\f$, the function
\begin{align}\label{Intr_Rem_IVPSol}
\tilde\uf := \uf + \int_0^t \e^{-(t-s)\elop_\bdopfull}\proj\f\,\ds - \e^{-t\elop_\bdopfull}\uf(0) 
\end{align} 
is the unique solution to \eqref{IVP}. The desired $\LR{p}$ estimates of $\uf$ follow from Theorem \ref{MainThm_HalfSpace}, while estimates of the two latter terms on the right-hand side in \eqref{Intr_Rem_IVPSol} follow by 
standard theory for analytic semi-groups; see for example \cite[Theorem 4.3.1]{LunardiBookSemigroups}. 
For more details, see also \cite[Theorem 5.1]{MaS17}. The connection between 
maximal regularity for parabolic initial-value problems and  
corresponding time-periodic problems was observed for the first time in the work of \textsc{Arendt} and \textsc{Bu} \cite[Theorem 5.1]{ArendtBu2002}. 
\end{rem}

\section{Preliminaries and Notation}\label{pre}

\subsection{Notation}

Unless otherwise indicated, $x$ denotes an element in $\R^n$ and $x':=(x_1,\ldots,x_{n-1})\in\R^{n-1}$.
The same notation is employed for $\xi\in\R^n$ and $\xi':=(\xi_1,\ldots,\xi_{n-1})\in\R^{n-1}$.

We denote by $\C_+:=\setc{z\in\C}{\impart(z)>0}$ and $\C_-:=\setc{z\in\C}{\impart(z)<0}$  the upper and lower complex plane, respectively.

The notation $\partial_j:=\partial_{x_j}$ is employed for partial derivatives with respect to spatial variables. Throughout, 
$\partial_t$ shall denote the partial derivative with respect to the time variable.
For a multi-index $\alpha\in\N^{n}$, we employ the notation 
$D^\alpha:=i^{|\alpha|}\partial_{x_1}^{\alpha_1}\ldots\partial_{x_n}^{\alpha_n}$.

We introduce the parabolic length 
\begin{align*}
\forall (\eta,\xi)\in\R\times\R^n:\quad \parnorm{\eta}{\xi}:=(|\eta|^2+|\xi|^{4m})^{\frac{1}{4m}}.
\end{align*}
We call a generic function $g$ \emph{parabolically $\alpha$-homogeneous} if $\lambda^\alpha g(\eta,\xi) = g(\lambda^{2m}\eta,\lambda\xi)$ for all $\lambda>0$.

\subsection{Paley-Wiener Theorem}

\begin{defn}\label{hardy_def}
The Hardy space $\Hardy_+^2\np{\R}$ consists of all functions $\f\in\LR{2}\np{\R}$ admitting a holomorphic extension to the lower complex plane 
$\tilde\f:\C_-\ra\C$ with
\begin{align*}
\sup_{y<0}\myint{\R}{|\tilde\f(x+iy)|^2}{x}<\infty,\quad \lim_{y\to 0-}\myint{\R}{|\tilde\f(x+iy)-f(x)|^2}{x}=0.
\end{align*}
The Hardy space $\Hardy_-^2\np{\R}$ consists of all functions $\f\in\LR{2}\np{\R}$ admitting a similar holomorphic extension to the upper complex plane. 
\end{defn}

\begin{prop}[Paley-Wiener Theorem]\label{paley_wiener_prop}
Let $\f\in \LR{2}\np{\R}$. 
Then $\supp f\subset \R_+$ if and only if $\ft{f}\in\Hardy_+^2$. 
Moreover, $\supp f\subset \R_-$ if and only if $\ft{f}\in\Hardy_-^2$.
\end{prop}
\begin{proof}
See for example \cite[Theorems VI.4.1 and VI.4.2]{Yos80}.
\end{proof}

\subsection{Time-periodic function spaces}

Let $\Omega\subset\wspace$ be a domain and
\begin{align*}
\CRciper(\R\times\Omega) := \setc{f\in\CRi(\R\times\Omega)}{f(t+\per,x)=f(t,x),\ f\in\CRci\bp{[0,\per]\times\Omega}}
\end{align*}
the space of smooth time-period functions with compact support in the spatial variable. Clearly,
\begin{align}
&\norm{f}_p := \Bp{\iper\int_0^\per\int_{\Omega}\snorm{f(t,x)}^p\,\dx\dt}^{\frac{1}{p}},\label{intro_defofLpNorm}\\ 
&\norm{f}_{1,2m,p} :=
\Bp{\norm{\partial_t f}_{p}^p +\sum_{0\leq \snorm{\alpha}\leq 2m} \norm{\partial_x^\alpha f}_{p}^p }^{\frac{1}{p}}\label{intro_defofSobNorm}
\end{align}
are norms on $\CRciper(\R\times\Omega)$. We define 
Lebesgue and anisotropic Sobolev spaces of time-periodic functions as completions
\begin{align*}
&\LRper{p}(\R\times\Omega):= \closure{\CRciper(\R\times\Omega)}{\norm{\cdot}_{p}},\quad
\WSRper{1,2m}{p}\bp{\R\times\Omega} := \closure{\CRciper(\R\times\overline{\Omega})}{\norm{\cdot}_{1,2m,p}}.
\end{align*}
One may identify 
\begin{align*}
\LRper{p}(\R\times\Omega) = \setc{\f\in\LRloc{p}(\R\times\overline{\Omega})}{\f(t+\per,x)=\f(t,x) \text{ for almost every } (x,t)}.
\end{align*}
On a similar note, one readily verifies that 
\begin{align*}
\WSRper{1,2m}{p}\np{\R\times\Omega}=\setc{\f\in\WSRloc{1,2m}{p}\np{\R\times\overline{\Omega}}}{\f(t+\per,x)=\f(t,x) \text{ for almost every } (x,t)},
\end{align*} 
provided $\Omega$ satisfies the segment condition.

We introduce anisotropic fractional order Sobolev spaces (Sobolev-Slobodecki\u{\i} spa\-ces) by real interpolation:
\begin{align*}
\WSRper{s,2ms}{p}(\R\times\Omega) = \bp{\LRper{p}(\R\times\Omega),\WSRper{1,2m}{p}\np{\R\times\Omega}}_{s,p}, \qquad s\in(0,1).
\end{align*}
For a $\CR{2m}$-smooth manifold $\Gamma\subset\R^n$, anisotropic Sobolev spaces $\WSRper{s,2ms}{p}(\R\times\Gamma)$ are defined in a similar manner. 
We can identify (see also Section \ref{gr} below) the trace space of $\WSRper{1,2m}{p}(\R\times\Omega)$ 
as $\WSRper{1-1/2mp,2m-1/p}{p}(\R\times\partial\Omega)$ in the sense that the trace operator maps the former continuously onto the latter.

\subsection{Function Spaces and the Torus Group Setting}\label{gr}

We shall further introduce a setting of function spaces in which the time axis $\R$ in the underlying domains is replaced with the torus $\torus:=\R/\per\Z$.
In such a setting, all functions are inherently $\per$-time-periodic. We shall therefore never have to verify
periodicity of functions \textit{a posteriori}, and it will always be clear in which sense the functions are periodic.

The setting of $\torus$-defined functions is formalized in terms of the canonical quotient mapping
$\quotientmap :\R\times\R^n \ra \torus \times \R^n,\ \quotientmap(t,x):=([t],x)$.
A differentiable structure on $\torus \times \R^n$ is inherited via the quotient mapping form $\R\times\R^n$. More specifically, 
for any domain $\Omega\subset\R^n$ we let
\begin{align*}
\CRi(\torus\times\Omega) := \setc{\uf:\torus\times\Omega\ra\C}{\uf\circ\quotientmap\in\CRi(\R\times\Omega)}
\end{align*}
and define for $\uf\in\CRi(\torus\times\Omega)$ derivatives by $\partial^\alpha\uf := \bp{\partial^\alpha \nb{\uf\circ\quotientmap}}\circ\quotientmap_{|[0,\per)\times\Omega}^{-1}$. We let
\begin{align*}
\CRci\np{\torus\times\Omega} := \setc{\uf\in\CRi(\torus\times\Omega)}{\supp\uf\text{ is compact}}
\end{align*}
denote the space of compactly supported smooth functions. Introducing the normalized Haar measure on $\torus$, we define norms $\norm{\cdot}_p$ and $\norm{\cdot}_{1,2m,p}$ on $\CRci\np{\torus\times\Omega}$ as in \eqref{intro_defofLpNorm}--\eqref{intro_defofSobNorm}.
The quotient mapping trivially respects derivatives and is isometric with respect to $\norm{\cdot}_p$ and $\norm{\cdot}_{1,2m,p}$. Letting
\begin{align*}
&\LR{p}(\torus\times\Omega):= \closure{\CRci(\torus\times\Omega)}{\norm{\cdot}_{p}},\quad
\WSR{1,2m}{p}\bp{\torus\times\Omega} := \closure{\CRci(\torus\times\overline{\Omega})}{\norm{\cdot}_{1,2m,p}},
\end{align*}
we thus obtain Lebesgue and Sobolev spaces that are isometrically isomorphic to the spaces 
$\LRper{p}(\R\times\Omega)$ and $\WSRper{1,2m}{p}\bp{\R\times\Omega}$, respectively.
Defining weak derivatives with respect to test functions $\CRci(\torus\times\Omega)$, one readily verifies that
\begin{align*}
\WSR{1,2m}{p}\bp{\torus\times\Omega} = \setc{\uf\in\LR{p}(\torus\times\Omega)}{\uf,\partial_t\uf,\partial_x^\alpha\uf\in\LR{p}(\torus\times\Omega)\text{ for all }\snorm{\alpha}\leq 2},
\end{align*}
provided $\Omega$ satisfies the segment property.

For $s\in(0,1)$, we define fractional ordered Sobolev spaces by real interpolation 
\begin{align*}
\WSR{s,2ms}{p}(\torus\times\Omega) = \bp{\LR{p}(\torus\times\Omega),\WSR{1,2m}{p}\np{\torus\times\Omega}}_{s,p},
\end{align*}
and thereby obtain spaces isometrically isomorphic to $\WSRper{s,2ms}{p}(\R\times\Omega)$.
In the half-space case, we clearly have 
\begin{align*}
\WSR{1,2m}{p}(\torus\times\R^n_+) = 
\LR{p}\bp{\R_+;\WSR{1,2m}{p}\np{\torus\times\R^{n-1}}}\cap 
\WSR{2m}{p}\bp{\R_+;\LR{p}\np{\torus\times\R^{n-1}}}.
\end{align*} 
Hence, for $l\in\N, l\leq 2m$ the trace operator
\begin{align}\label{DefOfTraceOpr}
\begin{aligned}
&\trace_l:\CRci(\torus\times\overline{\R^n_+})\ra\CRci(\torus\times{\R^{n-1}})^l,\\
&\trace_l(\uf)(t,x'):=\bp{\uf\np{t,x',0}, \partial_n\uf\np{t,x',0},\ldots,\partial_n^{l-1}\uf\np{t,x',0}}
\end{aligned}
\end{align} 
extends to a bounded operator 
\begin{align}\label{TraceOprMappingproperties}
\trace_l: \WSR{1,2m}{p}(\torus\times\R^n_+) \ra \prod_{j=0}^{l-1}\WSR{1-\frac{j}{2m}-\frac{1}{2mp},2m-j-\frac{1}{p}}{p}(\torus\times\R^{n-1}) 
\end{align}
that is onto; see for example \cite[Theorem 1.8.3]{Triebel_InterpolationTheory1978}. The existence of a bounded right inverse to 
$\trace_l$ can be shown by applying \cite[Theorem 2.9.1]{Triebel_InterpolationTheory1978}. 

We further introduce the operators
\begin{align}\label{intro_DefOfProjections}
\proj,\projcompl:\CRci(\torus\times\Omega)\ra\CRci(\torus\times\Omega),\quad \proj f := \int_\torus f(t,x)\,\dt,\quad \projcompl := \id - \proj,
\end{align}
which are clearly complementary projections. 
Since $\proj f$ is independent of the time variable $t\in\R$, we may at times treat $\proj f$ as a function of the space variable $x\in\Omega$ only. 
Both $\proj$ and $\projcompl$ extend to bounded operators on the Lebesgue space $\LR{p}(\torus\times\Omega)$ and Sobolev space
$\WSR{1,2m}{p}\bp{\torus\times\Omega}$. We employ the notation $\LRcompl{p}(\torus\times\Omega):=\projcompl\LR{p}(\torus\times\Omega)$
and $\WSRcompl{1,2m}{p}\bp{\torus\times\Omega}:=\projcompl\WSR{1,2m}{p}\bp{\torus\times\Omega}$ for the subspaces of $\projcompl$-invariant functions.
This notation is canonical extended to other spaces such interpolation spaces of Lebesgue and Sobolev spaces.  We sometimes refer to functions with $\f=\projcompl\f$
as \emph{purely oscillatory}.

Finally, we let 
\begin{align*}
\Dtracekonst_j  :=  1-\frac{j-1}{2m}-\frac{1}{2mp},\quad 
\Btracekonst_j  :=  1-\frac{m_j}{2m}-\frac{1}{2mp},\quad (j=1,\ldots,m)
\end{align*}
and put
\begin{align*}
&\tracespacecompl{\Dtracekonst}{p}(\torus\times\Omega):=\prod_{j=1}^m\WSRcompl{\Dtracekonst_j,2m\Dtracekonst_j}{p}(\torus\times\Omega),\quad \tracespacecompl{\Btracekonst}{p}(\torus\times\Omega):=\prod_{j=1}^m\WSRcompl{\Btracekonst_j,2m\Btracekonst_j}{p}(\torus\times\Omega).
\end{align*}

\subsection{Schwartz-Bruhat Spaces and Distributions}

When the spatial domain is the whole-space $\R^n$, we employ the notation ${\grp:=\torus\times\R^n}$. Equipped with the quotient topology via $\quotientmap$, $\grp$ becomes a locally compact abelian group. Clearly, the $\LR{p}(\grp)$ space corresponding to the Haar measure on $\grp$, appropriately normalized, coincides with the $\LR{p}\np{\torus\times\R^n}$ space introduced in the previous section.

We identify $\grp$'s dual group by $\dualgrp=\perf\Z\times\R^n$ by associating
$(k,\xi)\in\perf\Z\times\R^n$ with the character 
$\chi:\grp\ra\CNumbers,\ \chi(x,t):=\e^{ix\cdot\xi+ik t}$.
By default, $\dualgrp$ is equipped with the compact-open topology, which in this case coincides with the product of the 
discrete topology on $\perf\Z$ and the
Euclidean topology on $\R^n$.
The Haar measure on $\dualgrp$ is simply the product of the Lebesgue measure on $\R^n$ and the counting measure on $\perf\Z$. 

The Schwartz-Bruhat
space $\SR(\grp) $ of generalized Schwartz functions 
(originally introduced in \cite{Bruhat61}) can be described in terms of the semi-norms
\begin{align*}
\forall(\alpha,\beta,\gamma)\in\N_0\times\N_0^n\times\N_0^n:\quad 
\rho_{\alpha,\beta,\gamma}(\uf):=\sup_{(t,x)\in\grp} \snorm{x^\gamma\partial_t^\alpha\partial_x^\beta\uf(x,t)}
\end{align*}
as
\begin{align*}
\SR(\grp):=\setc{\uf\in\CRi(\grp)}{\forall(\alpha,\beta,\gamma)\in\N_0\times\N_0^n\times\N_0^n:\ \rho_{\alpha,\beta,\gamma}(\uf)<\infty}.
\end{align*}
The vector space $\SR(\grp)$ is endowed with the semi-norm 
topology.

The topological dual space $\TDR(\grp)$ of $\SR(\grp)$ is referred to as the space of tempered distributions on $\grp$. 
Observe that both $\SR(\grp)$ and $\TDR(\grp)$ remain closed under multiplication 
by smooth functions that have at most polynomial growth with respect to the spatial variables.
For a tempered distribution $\uf\in\TDR(\grp)$, distributional derivatives 
$\partial_t^\alpha\partial_x^\beta\uf\in\TDR(\grp)$ are defined by duality in the usual manner.
Also the support $\supp\uf$ is defined in the classical way. Moreover, we may restrict the distribution $\uf$ to a subdomain $\torus\times\Omega$ by
considering it as a functional defined only on the test functions from $\SR(\grp)$ supported in $\torus\times\Omega$.

A differentiable structure on $\dualgrp$ is obtained by introduction of the space 
\begin{align*}
\CRi(\dualgrp):=\setc{\wf\in\CR{}(\dualgrp)}{\forall k\in\perf\Z:\ \wf(k,\cdot)\in\CRi(\R^n)}.
\end{align*}
The Schwartz-Bruhat space on the dual group $\dualgrp$ is defined in terms of the semi-norms
\begin{align*}
\forall (\alpha,\beta,\gamma)\in\N_0\times\N_0^n\times\N_0^n:\ \dualrho_{\alpha,\beta,\gamma}(\wf):= 
\sup_{(\xi,k)\in\dualgrp} \snorm{k^\alpha \xi^\gamma \partial_\xi^\beta \wf(k,\xi)} 
\end{align*}
as
\begin{align*}
\begin{aligned} 
\SR(\dualgrp)&:=\setc{\wf\in\CRi(\dualgrp)}{\forall (\alpha,\beta,\gamma)\in\N_0\times\N_0^n\times\N_0^n:\ \dualrho_{\alpha,\beta,\gamma}(\wf)<\infty}.
\end{aligned}
\end{align*}
We also endow $\SR(\dualgrp)$ with the corresponding semi-norm topology and denote by $\TDR(\dualgrp)$ the topological dual space. 

\subsection{Fourier Transform}

As a locally compact abelian group, $\grp$ has a Fourier transform 
$\FT_{\grp}$ associated to it. The ability to utilize a Fourier transform that acts simultaneously in time $t\in\torus$ and space $x\in\R^n$ shall play a key role in the following.

The Fourier transform $\FT_\grp$ on $\grp$ is given by
\begin{align*}
\FT_\grp:\LR{1}(\grp)\ra\CR{}(\dualgrp),\quad \FT_\grp(\uf)(k,\xi):=\ft{\uf}(k,\xi):=
\int_\torus\int_{\R^n} \uf(t,x)\,\e^{-ix\cdot\xi-ik t}\,\dx\dt.
\end{align*}
If no confusion can arise, we simply write $\FT$ instead of $\FT_\grp$.
The inverse Fourier transform is formally defined by 
\begin{align*}
\iFT:\LR{1}(\dualgrp)\ra\CR{}(\grp),\quad \iFT(\wf)(t,x):=\ift{\wf}(t,x):=
\sum_{k\in\perf\Z}\,\int_{\R^n} \wf(k,\xi)\,\e^{ix\cdot\xi+ik t}\,\dxi.
\end{align*}
It is standard to verify that $\FT:\SR(\grp)\ra\SR(\dualgrp)$ is a homeomorphism with $\iFT$ as the actual inverse, provided the Lebesgue measure $\dxi$ is normalized appropriately. 
By duality, $\FT$ extends to a bijective mapping $\FT:\TDR(\grp)\ra\TDR(\dualgrp)$. 
The Fourier transform provides us with a calculus between the differential operators on $\grp$ and the 
polynomials on $\dualgrp$. As one easily verifies, for $\uf\in\TDR(\grp)$ and $(\alpha,\beta)\in\N_0\times\N_0^n$ we have
$\FT\bp{\partial_t^\alpha\partial_x^\beta\uf}=i^{\snorm{\alpha}+\snorm{\beta}}\,k^\alpha\,\xi^\beta\,\FT(\uf)$
as an identity in $\TDR(\dualgrp)$.

The projections introduced in \eqref{intro_DefOfProjections} can be extended trivially to projections on the Schwartz-Bruhat space 
$\proj,\projcompl:\SR(\grp)\ra\SR(\grp)$. 
Introducing the delta distribution $\delta_\Z$ on $\perf\Z$, that is, $\delta_\Z(k):=1$ if $k=0$ and $\delta_\Z(k):=0$ for $k\neq 0$,
we observe that 
$\proj\uf = \iFT_\grp\bb{\delta_\Z \FT_\grp\nb{\uf}}$ and 
$\projcompl\uf = \iFT_\grp\bb{(1-\delta_\Z) \FT_\grp\nb{\uf}}$.
Using these representations for $\proj$ and $\projcompl$, we naturally extend the projections to 
operators $\proj,\projcompl:\TDR(\grp)\ra\TDR(\grp)$. In accordance with the notation introduced above, we put
$\TDRcompl(\grp):=\projcompl\TDR(\grp)$.

In general, we shall utilize smooth functions $\multper\in\CRi(\dualgrp)$ with at most polynomial growth as Fourier multipliers by introducing 
the corresponding operator
\begin{align*}
\op[\multper]:\SR(\grp)\ra\TDR(\grp),\quad \op[\multper]\uf:=\FT^{-1}_\grp \bb{\multper\FT_\grp\nb{\uf}}.
\end{align*}  
We call $\multper$ an $\LR{p}(\grp)$-multiplier if $\op[\multper]$ extends to a bounded operator on $\LR{p}(\grp)$ for any $p\in(1,\infty)$.
The following lemmas provide us with criteria to determine if $\multper$ is an $\LR{p}(\grp)$-multiplier.

\begin{lem}\label{MultiplierLem_ZeroHomoMultipliers}
Let $\multper\in\CRi(\dualgrp)$. If $\multper=\multrestriction{\mult}{\dualgrp}$ for some parabolically $0$-homogeneous $\mult:\R\times\R^n\ra\C$, then $\op\nb{\multper}$ extends to a bounded operator $\op\nb{\multper}:\LR{p}(\grp)\ra\LR{p}(\grp)$. 
\end{lem}

\begin{proof}
The Transference Principle (established originally by de Leeuw \cite{dLe65} and later extended to a general setting of locally compact abelian groups
by Edwards and Gaudry \cite[Theorem B.2.1]{EdG77}),
makes it possible to ``transfer'' the investigation of Fourier multipliers from one group setting into another.
In our case, \cite[Theorem B.2.1]{EdG77} yields that $\multper$ is an $\LR{p}(\grp)$-multiplier, provided $\mult$ is an
$\LR{p}(\R\times\R^n)$-multiplier. To show the latter, we can employ one of the classical multiplier theorems available in 
the Euclidean setting. Since $\mult$ is parabolically $0$-homogeneous, it is easy to verify that $\mult$ meets for instance the conditions of the
Marcinkiewicz's multiplier theorem (\cite[Chapter IV, \S 6]{Stein70}). Thus, $\mult$ is an
$\LR{p}(\R\times\R^n)$-multiplier and by \cite[Theorem B.2.1]{EdG77} $\multper$ therefore an $\LR{p}(\grp)$-multiplier.  
\end{proof}

\begin{lem}\label{MultiplierLem_NegHomoMultipliers}
Let $\multper\in\CRi(\dualgrp\setminus\set{(0,0)})$ and $\alpha\leq 0$. If $\multper=\multrestriction{\mult}{\dualgrp}$ for some parabolically $\alpha$-homogeneous function $\mult:\R\times\R^n\setminus\set{(0,0)}\ra\C$, then $\op\nb{\multper}$ extends to a bounded operator $\op\nb{\multper}:\LRcompl{p}(\grp)\ra\LRcompl{p}(\grp)$. 
\end{lem}

\begin{proof}
Let $\cutoff\in\CRi(\R)$ be a ``cut-off'' function with 
$\cutoff(\eta)=0$ for $\snorm{\eta}<\frac{\pi}{\per}$ and 
$\cutoff(\eta)=1$ for $\snorm{\eta}\geq \perf$. 
Put $\Mult(\eta,\xi):=\cutoff(\eta)\mult(\eta,\xi)$.
Utilizing that $\mult$ is $\alpha$-homogeneous and $\alpha\leq 0$, one readily verifies that $\Mult$ satisfies
the conditions of Marcinkiewicz's multiplier theorem (\cite[Chapter IV, \S 6]{Stein70}). Consequently, $\Mult$ is an $\LR{p}(\R\times\R^n)$-multiplier.
For $\uf\in\LRcompl{p}(\grp)$, we have $\uf=\projcompl\uf$ and thus
\begin{align*}
\op\nb{\multper}\np{\uf} 
= \FT^{-1}_\grp \bb{\multper\FT_\grp\nb{\projcompl\uf}}
= \FT^{-1}_\grp \bb{\multper(1-\delta_\Z)\FT_\grp\nb{\projcompl\uf}}.
\end{align*}
Since $\multper(1-\delta_\Z)=\multrestriction{\Mult}{\dualgrp}$, we obtain from \cite[Theorem B.2.1]{EdG77} that 
$\multper(1-\delta_\Z)$ is an $\LR{p}\np{\grp}$-multiplier. Consequently, $\op\nb{\multper}\np{\uf}\leq\Ccn{C}\norm{\uf}_p$ for all $\uf\in\LRcompl{p}(\grp)$.  
\end{proof}

\begin{cor}\label{MultiplierLem}
Let $p\in(1,\infty)$ and $\beta\in[0,1]$. If $\Mmultiplier\in\CRi(\perf\Z\setminus\set{0}\times\R^{n})$ is parabolically $0$-homogeneous, then   
$\op[M]$ extends to a bounded operator $\op[M]:\WSRcompl{\beta,2m\beta}{p}(\grp)\ra\WSRcompl{\beta,2m\beta}{p}(\grp)$.
\end{cor}

\subsection{Time-Periodic Bessel Potential Spaces}

Time-periodic Bessel Potential spaces can be defined via the Fourier transform $\FT_\grp$. We shall only introduce Bessel Potential spaces of purely oscillatory distributions:
\begin{align*}
\HSRcompl{s}{p}:=\setc{\f\in\TDRcompl(\grp)}{\op\nb{\parnorm{k}{\xi}^s}{\f}\in\LR{p}\np{\grp}}\quad\text{for } s\in\R,\,p\in(1,\infty). 
\end{align*}
Utilizing Lemma \ref{MultiplierLem_NegHomoMultipliers}, one readily verifies that $\HSRcompl{s}{p}$ is a Banach space with respect to the norm 
\begin{align*}
\norm{\f}_{s,p}:= \norm{\op\nb{\parnorm{k}{\xi}^s}{\f}}_p.
\end{align*}
Time-periodic Bessel Potential spaces on the half-space are defined via restriction of distributions in the time-periodic Bessel Potential spaces defined above:
\begin{align*}
&\HSRcompl{s}{p}(\halfgrp):=\setc{\funcrestriction{f}{\halfgrp}}{f\in\HSRcompl{s}{p}},\\
&\norm{f}_{s,p,\halfgrp}:=
\inf\setc{\norm{F}_{s,p}}{F\in\HSRcompl{s}{p},\ \funcrestriction{F}{\halfgrp}=f}.
\end{align*}
Identifying $\HSRcompl{s}{p}(\halfgrp)$ as a factor space of $\HSRcompl{s}{p}$ in the canonical way, we see that $\HSRcompl{s}{p}(\halfgrp)$ is a Banach
space in the norm $\norm{\cdot}_{s,p,\halfgrp}$.

\begin{prop}\label{TPBesselEqualsTPSobWholespacecaseProp}
Let $p\in(1,\infty)$. Then $\HSRcompl{2m}{p}(\grp)=\WSRcompl{1,2m}{p}(\grp)$ and $\HSRcompl{2m}{p}(\halfgrp)=\WSRcompl{1,2m}{p}(\halfgrp)$
with equivalent norms. 
\end{prop}

\begin{proof}
It follows from Lemma \ref{MultiplierLem_NegHomoMultipliers} that $\op\bb{\parnorm{\eta}{\xi}^{-m}}$ extends to a bounded operator on $\LRcompl{p}(\grp)$, which implies that $\norm{\f}_{m,p}$ is equivalent to the norm $\norm{f}_p+\norm{\f}_{m,p}$. From this, we infer that 
$\HSRcompl{2m}{p}=\WSRcompl{1,2m}{p}(\grp)$. A standard method (see for example 
\cite[Theorem 4.26]{adams:sobolevspaces}) can be used to construct an extension operator $\Ext:\WSRcompl{1,2m}{p}(\halfgrp)\ra\WSRcompl{1,2m}{p}(\grp)$.
The existence of an extension operator combined with the fact that $\HSRcompl{2m}{p}=\WSRcompl{1,2m}{p}(\grp)$ implies 
$\HSRcompl{2m}{p}(\halfgrp)=\WSRcompl{1,2m}{p}(\halfgrp)$.   
\end{proof}

\begin{prop}\label{TPBesselEqualsTPSobHalfspacecaseProp}
Let $s\in\R$. Then
\begin{align}
&\norm{\uf}_{s+1,p,\halfgrp}\leq \Ccn{C} \bp{ \norm{\partial_n\uf}_{s,p,\halfgrp} + \norm{\op\bb{{\parnorm{k}{\xip}}}\uf}_{s,p,\halfgrp}},\label{TPBesselEqualsTPSobHalfspacecasePropEst}\\
&\norm{\partial_j\uf}_{s,p,\halfgrp}\leq  { \norm{\uf}_{s+1,p,\halfgrp}}\quad(j=1,\ldots,n)\label{TPBesselEqualsTPSobHalfspacecasePropEst2}\\
&\norm{\partial_t\uf}_{s,p,\halfgrp}\leq  { \norm{\uf}_{s+2m,p,\halfgrp}}\label{TPBesselEqualsTPSobHalfspacecasePropEst3}
\end{align}
for all $\uf\in\TDR(\grp)$.
\end{prop}

\begin{proof}
The complex function $z\mapsto(iz+\pkxp)^{-1}$ is holomorphic in the lower complex plane.
Due to Lemma \ref{MultiplierLem_ZeroHomoMultipliers}, we can employ Proposition \ref{paley_wiener_prop} to
conclude
\begin{align}\label{TPBesselEqualsTPSobHalfspacecaseProp_SuppProperty}
\supp\psi\subset\lhalfgrpclosed \quad \Ra \quad \supp\Bp{\op\bb{(i\xi_n+\pkxp)^{-1}}\psi}\subset\lhalfgrpclosed 
\end{align} 
for all $\psi\in\SRcompl(\grp)$. By duality, the same is true for all $\psi\in\TDRcompl(\grp)$. 
We employ Lemma \ref{MultiplierLem_ZeroHomoMultipliers} to estimate
\begin{align*}
\norm{\uf}_{s+1,p,\halfgrp} 
&\leq \Ccn{C}\,\inf\setcl{\norml{\op\bb{i\xi_n+\pkxp}U}_{s,p}}{U\in\HSRcompl{s+1}{p},\ \funcrestriction{U}{\halfgrp}=\funcrestriction{\uf}{\halfgrp}}.
\end{align*} 
It follows from \eqref{TPBesselEqualsTPSobHalfspacecaseProp_SuppProperty} that
\begin{align*}
\funcrestriction{U}{\halfgrp}=\funcrestriction{\uf}{\halfgrp} \quad \iff\quad 
\funcrestriction{\op\bb{i\xi_n+\pkxp}U}{\halfgrp}=\funcrestriction{\op\bb{i\xi_n+\pkxp}\uf}{\halfgrp}. 
\end{align*}
We thus conclude 
\begin{align*}
\norm{\uf}_{s+1,p,\halfgrp} 
&\leq \Ccn{C}\, \norm{\op\bb{i\xi_n+\pkxp}\uf}_{s,p,\halfgrp}
\end{align*} 
and thereby \eqref{TPBesselEqualsTPSobHalfspacecasePropEst}. Furthermore,
\begin{align*}
\norm{\partial_n\uf}_{s,p,\halfgrp}  
&=\inf\setc{\norm{U}_{s,p}}{U\in\HSRcompl{s}{p},\ U=\partial_n\uf\ \tin\halfgrp}\\
&\leq\inf\setc{\norm{\partial_n V}_{s,p}}{V \in\HSRcompl{s+1}{p},\ V=\uf\ \tin\halfgrp}\\
&\leq\inf\setc{\norm{V}_{s+1,p}}{V \in\HSRcompl{s+1}{p},\ V=\uf\ \tin\halfgrp} = \norm{\uf}_{s+1,p,\halfgrp},
\end{align*}
where the last inequality follows by an application of Lemma \ref{MultiplierLem_NegHomoMultipliers}.
We have thus shown \eqref{TPBesselEqualsTPSobHalfspacecasePropEst2}. One may verify \eqref{TPBesselEqualsTPSobHalfspacecasePropEst3} in 
a similar manner. 
\end{proof}

\begin{lem}\label{BS_SobSloboMultiplierLem}
Let $\beta\in(0,1)$ and $\alpha\in\bp{2m(\beta-1),2m\beta}$. Then $\op\bb{\pkx^\alpha}$ extends to a bounded operator
$\op\bb{\pkx^\alpha}:\WSRcompl{\beta,2m\beta}{p}(\torus\times\R^n)\ra\WSRcompl{\beta-\frac{\alpha}{2m}}{2m\beta-\alpha}(\torus\times\R^n)$.
\end{lem}

\begin{proof}
By interpolation, we directly obtain that $\op\bb{\pkx^\alpha}$ extends to a bounded operator
\begin{align*}
\op\bb{\pkx^\alpha}:\bp{\HSRcompl{\alpha}{p},\HSRcompl{2m}{p}}_{\theta,p}\ra
\bp{\LRcompl{p}(\torus\times\R^n),\HSRcompl{2m-\alpha}{p}}_{\theta,p}
\end{align*}
for any $\theta\in(0,1)$. Choose $\theta :=\frac{2m\beta-\alpha}{2m-\alpha}$. Using a dyadic decomposition of the Fourier space with respect to the parabolic length, $\HSRcompl{s}{p}$ can be identified as the complex interpolation space $[\HSRcompl{s_0}{p},\HSRcompl{s_1}{p}]_{\omega}$ with $\frac{1}{s}=\frac{1-\omega}{s_0}+\frac{\omega}{s_1}$ by verifying that it is a retract of $L^p(l^{s,2})$ as in \cite[Theorem 6.4.3]{BergLoefstroemBook}. In fact, relying on the Transference Principle, we have a Mikhlin's multiplier theorem at our disposal, which is the key ingredient in \cite[Theorem 6.4.3]{BergLoefstroemBook}. Hence, by reiteration and Proposition \ref{TPBesselEqualsTPSobWholespacecaseProp}
\begin{align*}
\bp{\HSRcompl{\alpha}{p},\HSRcompl{2m}{p}}_{\theta,p} = 
\bp{\LRcompl{p}(\torus\times\R^n),\HSRcompl{2m}{p}}_{\beta,p}
=\WSRcompl{\beta,2m\beta}{p}(\torus\times\R^n).
\end{align*}
Similarly, 
$\bp{\LRcompl{p}(\torus\times\R^n),\HSRcompl{2m-\alpha}{p}}_{\theta,p}=\WSRcompl{\beta-\frac{\alpha}{2m}}{2m\beta-\alpha}(\torus\times\R^n)$. 
\end{proof}

Finally, we 
characterize the trace spaces of the time-periodic Bessel Potential spaces.

\begin{lem}\label{BS_TraceSpaceBesselPotentialLem}
Let $p\in(1,\infty)$.
The trace operator $\trace_m$ defined in \eqref{DefOfTraceOpr} extends to a bounded operator
\begin{align}\label{BS_TraceSpaceBesselPotentialLemTraceOpr}
\trace_m:\HSRcompl{m}{p}(\halfgrp)\ra
\prod_{j=0}^{m-1}\WSR{\half-\frac{j}{2m}-\frac{1}{2mp},m-j-\frac{1}{p}}{p}(\torus\times\R^{n-1})
\end{align}
that is onto and has a bounded right inverse. If $\uf\in\HSRcompl{m}{p}(\grp)$ with $\supp(\uf)\subset\halfgrpclosed$, then 
$\trace_m\bp{\funcrestriction{\uf}{\halfgrp}}=0$. If $\uf\in\HSRcompl{m}{p}(\halfgrp)$ with $\trace_m\bp{\funcrestriction{\uf}{\halfgrp}}=0$, then 
$\uf$ is the restriction of a function $U\in\HSRcompl{m}{p}(\grp)$ with $\supp U\subset\halfgrpclosed$.
\end{lem}

\begin{proof}
For either $I=\R$ or $I=\R_+$, put
\begin{align*}
\VR(I):=\LR{p}\bp{I;\HSRcompl{m}{p}(\torus\times\R^{n-1})}\cap\HSR{m}{p}\bp{I;\LRcompl{p}(\torus\times\R^{n-1})}.
\end{align*}
We first verify that 
$\HSRcompl{m}{p}=\VR(\R)$ with equivalent norms.
It is straightforward to obtain the embedding $\HSRcompl{m}{p}\embeds \VR(\R)$. To show the reverse embedding, consider $\uf\in \VR(\R)$.
Then
$\norm{\op\nb{\parnorm{k}{\xip}^m}\uf}_p\leq \norm{\uf}_{\VR(\R)}$ and
$\norm{\op\nb{\xi_n^m}\uf}_p\leq \norm{\uf}_{\VR(\R)}$. By Lemma \ref{MultiplierLem_ZeroHomoMultipliers},
\begin{align*}
\multper:\grp\ra\C,\quad\multper(k,\xi):=\frac{\parnorm{k}{\xi}^m}{\parnorm{k}{\xip}^m+\xi_n^m}
\end{align*}
is an $\LR{p}(\grp)$ multiplier. It follows that 
$\norm{\op\nb{\parnorm{k}{\xi}^m}\uf}_p\leq \norm{\uf}_{\VR(\R)}$ and thus the embedding $\VR(\R)\embeds\HSRcompl{m}{p}$.
We conclude $\HSRcompl{m}{p}=\VR(\R)$. It is standard to show existence of an extension operator $\VR(\R_+)\ra\VR(\R)$; see for example 
\cite[Lemma 2.9.1]{Triebel_InterpolationTheory1978}. By restriction to $\halfgrp$, it thus follows that $\HSRcompl{m}{p}(\halfgrp)=\VR(\R_+)$.
The classical trace method now implies that trace operator extends to a bounded operator 
\begin{align*}
\trace_m: \HSRcompl{m}{p}(\halfgrp)\ra 
\prod_{j=0}^{m-1}\bp{\LRcompl{p}(\torus\times\R^{n-1}),\HSRcompl{m}{p}(\torus\times\R^{n-1})}_{1-\frac{j}{m}-\frac{1}{mp},p} 
\end{align*}  
that is onto; see for example \cite[Theorem 1.8.3]{Triebel_InterpolationTheory1978}. 
The existence of a bounded right inverse can be shown as in \cite[Theorem 2.9.1]{Triebel_InterpolationTheory1978}. Again by reiteration
we identify
\begin{align*}
&\bp{\LRcompl{p}(\torus\times\R^{n-1}),\HSRcompl{m}{p}(\torus\times\R^{n-1})}_{1-\frac{j}{m}-\frac{1}{mp},p} \\
&\ =\bp{\LRcompl{p}(\torus\times\R^{n-1}),\HSRcompl{2m}{p}(\torus\times\R^{n-1})}_{\half-\frac{j}{2m}-\frac{1}{2mp},p}
=\WSR{\half-\frac{j}{2m}-\frac{1}{2mp},m-j-\frac{1}{p}}{p}(\torus\times\R^{n-1}).
\end{align*} 
Thus, we conclude \eqref{BS_TraceSpaceBesselPotentialLemTraceOpr}.

Consider now 
$\uf\in\HSRcompl{m}{p}(\grp)$ with $\supp(\uf)\subset\halfgrpclosed$. As above we can identify $\uf$ as an element of $\VR(\R)$, which necessarily satisfies
$\uf(0)=0$. It follows that $\trace_m\bp{\funcrestriction{\uf}{\halfgrp}}=0$.
If on the other hand $\uf\in\HSRcompl{m}{p}(\halfgrp)$ with $\trace_m\bp{\funcrestriction{\uf}{\halfgrp}}=0$, then it is standard to show that
$\uf$ can be approximated by a sequence of functions from $\CRci\np{\halfgrp}$; see for example \cite[Theorem 2.9.1]{Triebel_InterpolationTheory1978}.
Clearly, this sequence also converge in $\HSRcompl{m}{p}(\grp)$. The limit function $U\in\HSRcompl{m}{p}(\grp)$ satisfies $\supp U\subset\halfgrpclosed$
and $\funcrestriction{U}{\halfgrp}=\uf$.
\end{proof}


\section{Constant Coefficients in the Whole- and Half-Space}

In this section, we establish the assertion of Theorem \ref{MainThm_HalfSpace}.
We first treat the whole-space case, and then show the theorem as stated in the half-space case.
Since we consider only the differential operators with constant coefficients in this section, we employ the simplified notation
$\elopfull(D)$ instead of $\elopfull(x,D)$. Replacing the differential operator $D$ with $\xi\in\R^n$, we refer to $\elopfull(\xi)$
as the symbol of $\elopfull(D)$. 

\subsection{The Whole Space}\label{WholeSpaceSection}

We consider first the case of the spatial domain being the whole-space $\R^n$. 
The time-space domain then coincides with the locally abelian group $\grp$, and we can thus employ the Fourier transform $\FT_\grp$
and base the proof on an investigation of the corresponding Fourier multipliers.
 
\begin{lem}\label{wholeop_bdd_lem}
Assume that $\elop$ is properly elliptic and satisfies 
Agmon's condition on the two rays $\e^{i\theta}$ with $\theta=\pm\frac{\pi}{2}$.
Let $p\in\np{1,\infty}$, $s\in\R$ and 
\begin{align}\label{wholeop_bdd_lem_FparopperDef}
&\Fparopper: \dualgrp\ra\C, \qquad \Fparopper(k,\xi):=ik+\elop(\xi).
\end{align}
Then the linear operators
\begin{align*}
\parop:=\op[\Fparopper]:\SR_\bot(\grp)\to \SR'_\bot(\grp),\quad 
\paropinv:=\op[\Fparopper^{-1}]:\SR_\bot(\grp)\to \SR'_\bot(\grp)
\end{align*} 
extend uniquely to bounded operators 
\begin{align}\label{wholeop_bdd_lem_Mappingproperties}
\parop:\HSR{s}{p}_\bot\np{\grp}\to\HSR{s-2m}{p}_\bot\np{\grp},\quad 
\paropinv:\HSR{s-2m}{p}_\bot\np{\grp}\to\HSR{s}{p}_\bot\np{\grp}.
\end{align}
In the setting \eqref{wholeop_bdd_lem_Mappingproperties}, $\paropinv$ is the actual inverse of $\parop$.
\end{lem}

\begin{proof}
Let
\begin{align*}
\mult:\R\times\R^n\minusnullset\ra\C,\quad \mult (\eta,\xi)&:=\frac{i\eta+\elop(\xi)}{\pex^{2m}}.
\end{align*}
Clearly, $\mult$ is parabolically $0$-homogeneous. By Lemma \ref{MultiplierLem_NegHomoMultipliers}, it follows that $\multper:=\funcrestriction{\mult}{\dualgrp}$ is an $\LR{p}(\grp)$-multiplier. 
Since $\Fparopper={\pkx^{2m}\multper(k,\xi)}$, we conclude that
\begin{align*}
\norm{\parop\f}_{s-2m,p}&=\norm{\op\bb{\pkx^{2m}\multper(k,\xi)\f}}_{s-2m,p} = \norm{\op\bb{\multper(k,\xi)}\f}_{s,p}
\leq \Ccn{C}\norm{\f}_{s,p}.
\end{align*}
Since $\elop$ 
satisfies Agmon's condition for $\theta=\pm\frac{\pi}{2}$, it follows that
$\elop(\xi)\notin i\R$ for all $\xi\in\R^n\setminus\set{0}$. Consequently,
$\mult^{-1}:\R\times\R^n\minusnullset\ra\C$ is well-defined and clearly parabolically $0$-homogeneous. We deduce as above that 
\begin{align*}
\norm{\paropinv\f}_{s,p}&=\norm{\op\bb{\pkx^{-2m}\multper(k,\xi)^{-1}\f}}_{s,p} = \norm{\op\bb{\multper(k,\xi)^{-1}}\f}_{s-2m,p}
\leq \Ccn{C}\norm{\f}_{s,p}.
\end{align*}
Consequently, $\parop$ and $\paropinv$ extend uniquely to bounded operators  $\parop:\HSR{s}{p}_\bot\to\HSR{s-2m}{p}_\bot$ 
and $\paropinv:\HSR{s-2m}{p}_\bot\to\HSR{s}{p}_\bot$, respectively. 
Clearly, $\paropinv$ is the actual inverse of $\parop$.
\end{proof}

\begin{thm}\label{MainThm_ws}
Assume  $\elop$ is properly elliptic and satisfies Agmon's condition on the two rays $\e^{i\theta}$ with $\theta=\pm\frac{\pi}{2}$.
Let $s\in\R$ and $p\in (1,\infty)$. There is a constant $\Ccn{C}>0$ such that 
\begin{align}\label{TPADNEst_ws}
\norm{\uf}_{s,p}\leq \Ccn{C} \Bp{\norm{\partial_t\uf+\elopfull\uf}_{s-2m,p}+ \norm{u}_{s-1,p}}.
\end{align}
for all $\uf\in\HSR{s}{p}_\bot(\grp)$. 
\end{thm}

\begin{proof}
Since $\parop\uf= \partial_t\uf + \elop\uf$, we employ Lemma \ref{wholeop_bdd_lem} to estimate
\begin{align*}
\norm{\uf}_{s,p} \leq 
\Ccn{C}\norm{\partial_t\uf + \elop\uf}_{s-2m,p} 
\leq \Ccn{C}\norm{\partial_t\uf + \elopfull\uf}_{s-2m,p} + \norm{\bb{\elopfull-\elop}\uf}_{s-2m,p}. 
\end{align*}
Since the differential operator $\elopfull-\elop$ contains derivatives of at most of order $2m-1$, 
we conclude \eqref{TPADNEst_ws} by a similar multiplier argument as in the proof of Lemma \ref{wholeop_bdd_lem}.
\end{proof}

\subsection{The Half Space with Dirichlet Boundary Condition}\label{HalfSpaceDirichletBoundaryCondSection}

In the next step, we consider the case of the spatial domain being the half-space $\halfspace$ and boundary operators corresponding to Dirichlet boundary 
conditions. 
As in the whole-space case, we shall work with the symbol of $\partial_t+\elop$. In the following lemma, we collect its key properties.

\begin{lem}\label{SymbolPropertiesLem}
Assume  $\elop$ is properly elliptic and satisfies Agmon's condition on the two rays $\e^{i\theta}$ with $\theta=\pm\frac{\pi}{2}$. Let 
\begin{align*}
\Fparop:\R\times\R^n\ra\C,\quad \Fparop(\eta,\xi',\xi_n):=i\eta+\elop(\xi',\xi_n).
\end{align*}
\begin{enumerate}[(1)]
\item\label{SymbolPropertiesLem_PropNumberOfRoots} For every $(\eta,\xi')\in{\R\times\R^{n-1}\minusnullset}$ the complex polynomial 
$z\mapsto\Fparop(\eta,\xi',z)$ has exactly 
$m$ roots $\rho_j^+(\eta,\xi')\in\C_-$ in the upper complex plane, and 
$m$ roots $\rho_j^-(\eta,\xi')\in\C_+$ in the lower complex plane ($j\in\set{1,\ldots,m})$.
\item The functions
\begin{align}\label{FparopDef}
\begin{aligned}
&\Fparop_\pm:\R\times\R^n\setminus\setc{(\eta,\xi)}{(\eta,\xip)=(0,0)} \ra\C,\\
&\Fparop_\pm(\eta,\xi):=\prod_{j=1}^{m}\bp{\xi_n-\rho_j^\pm(\eta,\xi')}
\end{aligned}
\end{align}
are parabolically $m$-homogeneous.
\item The coefficients of the polynomials
$z\mapsto\Fparop_\pm(\eta,\xi',z)$, more specifically the functions
$\Fparopcoef^\pm_\alpha: \R\times\R^{n-1}\minusnullset\ra\C$ $(\alpha=0,\ldots,m)$ 
with the property that 
\begin{align}\label{SymbolPropertiesLem_FparopAsSumWithCoefficients}
\Fparop_\pm(\eta,\xi',\xin) = \sum_{\alpha=0}^m \Fparopcoef^\pm_\alpha(\eta,\xip)\xin^{m-\alpha},
\end{align}
are analytic. Moreover, $\Fparopcoef^\pm_\alpha$ is parabolically $\alpha$-homogeneous.
\end{enumerate}
\end{lem}

\begin{proof}\mbox{}
\begin{enumerate}[(1)]
\item Since $\elop$ is properly elliptic, the polynomial $z\mapsto\Fparop(0,\xi',z)$ has exactly m roots in the upper and lower complex plane, respectively. 
Recall that $\elop(x,\xi)\notin i\R$ for all $\xi\in\R^n\setminus\set{0}$. Since the roots of a polynomial depend continuously on the polynomial's coefficients, we deduce part (\ref{SymbolPropertiesLem_PropNumberOfRoots}) of the lemma.
\item Since $\Fparop$ is parabolically $2m$-homogeneous, the roots $\rho_j^\pm$ are parabolically $1$-homo\-geneous. It follows that $\Fparop_\pm$
is parabolically $m$-homogeneous. 
\item The analyticity of the coefficients $\Fparopcoef^\pm_\alpha$ follows by a classical argument; see for example \cite[Chapter 4.4]{TanabeBook}.
The coefficient $\Fparopcoef^\pm_\alpha$ being parabolically $\alpha$-homogeneous is a direct consequence of $\Fparop_\pm$ being $m$-homogeneous. \qedhere
\end{enumerate}
\end{proof}

\begin{lem}\label{halfop_bdd_lem}
Assume  $\elop$ is properly elliptic and satisfies Agmon's condition on the two rays $\e^{i\theta}$ with $\theta=\pm\frac{\pi}{2}$.
Put $\Fparopper_\pm:=\Fparop_\pm|_{\dualgrp}$, where $\Fparop_\pm$ is defined by \eqref{FparopDef}. 
 Let $p\in\np{1,\infty}$ and $s\in\R$. Then the linear operators
\begin{align*}
\paroppm:=\op[\Fparopper_\pm]:\SR_\bot(\grp)\to \SR'_\bot(\grp),\quad
\paroppminv:=\op[\Fparopper_\pm^{-1}]:\SR_\bot(\grp)\to \SR'_\bot(\grp)
\end{align*}
 extend uniquely to bounded and mutually inverse operators $\paroppm:\HSR{s}{p}_\bot(\grp)\to\HSR{s-m}{p}_\bot(\grp)$ and 
$\paroppminv:\HSR{s-m}{p}_\bot(\grp)\to\HSR{s}{p}_\bot(\grp)$, respectively.
\end{lem}

\begin{proof}
The assertion of the lemma follows as in the proof of Lemma \ref{wholeop_bdd_lem}, provided we can show that the restriction to $\dualgrp$ of the
multiplier  
\begin{align*}
\mult:\R\times\R^n\setminus\setc{(\eta,\xi)}{(\eta,\xip)=(0,0)}\ra\C,\quad  \mult(\eta,\xi)&:=\frac{\Fparop_\pm(\eta,\xi)}{\pex^{m}}
\end{align*}
and its inverse are $\LRcompl{p}(\grp)$-multipliers. Although $\mult$ is parabolically $0$-homogeneous, we cannot apply 
Lemma \ref{MultiplierLem_NegHomoMultipliers} directly since $\mult$ is not defined on all of $\R\times\R^n\setminus\set{(0,0)}$.
Instead, we recall \eqref{SymbolPropertiesLem_FparopAsSumWithCoefficients} and observe that 
\begin{align}\label{halfop_bdd_lem_FparopSeparationTrick}
\mult(\eta,\xi) = \sum_{\alpha=0}^m \frac{ \Fparopcoef^\pm_\alpha(\eta,\xip)}{\pexp^\alpha} \cdot \frac{\xi_n^{m-\alpha}\pexp^\alpha}{\pex^m}
=:\sum_{\alpha=0}^m\mult_1^\alpha \cdot \mult_2^\alpha.
\end{align}
Owing to the $\alpha$-homogeneity of $\Fparopcoef^\pm_\alpha$, Lemma \ref{MultiplierLem_NegHomoMultipliers} yields that both 
$\funcrestriction{{\mult_1^\alpha}}{\dualgrp}$ and $\funcrestriction{{\mult_2^\alpha}}{\dualgrp}$ are  $\LR{p}(\grp)$-multipliers. Consequently, also $\mult$ 
is an $\LR{p}(\grp)$-multiplier, and we thus conclude as in the proof of Lemma \ref{wholeop_bdd_lem} that $\paroppm$
extends uniquely to a bounded operator $\paroppm:\HSR{s}{p}_\bot(\grp)\to\HSR{s-m}{p}_\bot(\grp)$.

To show the corresponding property for $\paroppminv$, we introduce a cut-off function $\cutoff\in\CRi(\R)$  with 
$\cutoff(\eta)=0$ for $\snorm{\eta}<\frac{\pi}{\per}$ and $\cutoff(\eta)=1$ for $\snorm{\eta}\geq \perf$.
We claim that 
\begin{align*}
\tmult:\R\times\R^n\ra\C,\quad  \tmult(\eta,\xi):= \cutoff(\eta) \frac{\pex^m}{\Fparop_\pm(\eta,\xi)}
\end{align*}
is an $\LR{p}\np{\R\times\R^n}$-multiplier. Indeed, utilizing that $\Fparop_\pm$ is $m$-homogeneous, we see 
that $\Fparop_\pm$ can be bounded below by
\begin{align}\label{halfop_bdd_lem_FparopBoundBelow}
\snorm{\Fparop_\pm(\eta,\xi)} \geq \pex^m\, \inf_{\parnorm{\teta}{\txi}=1,(\teta,\txip)\neq(0,0)} \snorm{\Fparop_\pm(\teta,\txi)},
\end{align}
where the infimum above is strictly positive due to the roots in definition \eqref{FparopDef} satisfying $\lim_{(\eta,\xip)\ra(0,0)}\rho_j^\pm(\eta,\xip)=0$. Using only \eqref{halfop_bdd_lem_FparopBoundBelow} and 
the $\alpha$-homogeneity of the coefficients $\Fparopcoef^\pm_\alpha$ as in \eqref{halfop_bdd_lem_FparopSeparationTrick}, it is now straightforward to verify that $\tmult$ satisfies the condition of the 
Marcinkiewicz's multiplier theorem (\cite[Chapter IV, \S 6]{Stein70}). Thus, $\tmult$ is an
$\LR{p}(\R\times\R^n)$-multiplier and by \cite[Theorem B.2.1]{EdG77} $\funcrestriction{{\tmult}}{{\dualgrp}}$ therefore an $\LR{p}(\grp)$-multiplier.
Since the restriction of the corresponding operator $\op\bb{\funcrestriction{{\tmult}}{{\dualgrp}}}:\LRcompl{p}(\grp)\ra\LRcompl{p}(\grp)$ to the subspace of purely periodic functions coincides with $\op\bb{\funcrestriction{{\mult^\inverse}}{{\dualgrp}}}:\LRcompl{p}(\grp)\ra\LRcompl{p}(\grp)$,
we deduce as in the proof of Lemma \ref{wholeop_bdd_lem} that $\paroppminv$ extends uniquely to a bounded operator
$\paroppminv:\HSR{s-m}{p}_\bot(\grp)\to\HSR{s}{p}_\bot(\grp)$.
\end{proof}

The lemma above provides us with at decomposition of the differentiable operators in \eqref{wholeop_bdd_lem_Mappingproperties}, that is, for 
$\parop:\HSR{s}{p}_\bot(\grp)\to\HSR{s-2m}{p}_\bot(\grp)$ and $\paropinv:\HSR{s-2m}{p}_\bot(\grp)\to\HSR{s}{p}_\bot(\grp)$ the decompositions $\parop=\paropp\paropm=\paropm\paropp$ and $\paropinv=\paroppinv\paropminv=\paropminv\paroppinv$
are valid provided $\parop$ is normalized accordingly.  Employing the Paley-Wiener Theorem, 
we shall now show that the operators $\parop_\pm$ and $\paropinv_\pm$ ``respect'' the support of a function in the upper (lower) half-space.

\begin{lem}\label{paley_wiener_lem}
Assume  $\elop$ is properly elliptic and satisfies Agmon's condition on the two rays $\e^{i\theta}$ with $\theta=\pm\frac{\pi}{2}$.
 Let $p\in\np{1,\infty}$, $s\in\R$ and consider $\uf\in \HSR{s}{p}_\bot(\grp)$.
 \begin{enumerate}
  \item\label{paley_wiener_lem_i} If $\supp\uf\subset\halfgrpclosed$, then $\supp\paropp\uf\subset\halfgrpclosed$ and $\supp\paroppinv\uf\subset\halfgrpclosed$.
  \item\label{paley_wiener_lem_ii} If $\supp\uf\subset\lhalfgrpclosed$, then $\supp\paropm\uf\subset\lhalfgrpclosed$ and $\supp\paropminv\uf\subset\lhalfgrpclosed$.
 \end{enumerate}
\end{lem}
\begin{proof}
 We shall prove only part \eqref{paley_wiener_lem_i}, for part \eqref{paley_wiener_lem_ii} follows analogously. We employ the notation $\grpprime:=\torus\times\R^{n-1}$ and the canonical decomposition $\FT_\grp=\FT_\grpprime\FT_\R$ of the Fourier transform. In view of Lemma \ref{halfop_bdd_lem}, it suffices to consider only $\uf\in \SR(\grp)$ with $\supp\uf\subset\halfgrpclosed$.
 
 For fixed $\kt\in\perf\Z\setminus\set{0}$ and $\xi'\in\R^{n-1}$, we let $\mathsf{D}(\kt,\xi'):=\FT_{\R}^{-1}\Fparopper_+(\kt,\xi',\cdot)\FT_{\R}$. Since $\Fparopper_+$ is a polynomial with respect to the variable $\xi_n$, $\mathsf{D}(\kt,\xi')$ is a differential operator in $x_n$ and hence $\supp(\mathsf{D}(\kt,\xi')\f)\subset \overline{\R_+}$ for every $\f\in\SR'(\R)$ with $\supp\f\subset\overline{\R_+}$. 
Clearly, $\supp([\FT_\grpprime\uf](\kt,\xi',\cdot))\subset\overline{\R_+}$.
Since $\FT_\grpprime[\paropp\uf](\kt,\xi',\cdot)=[\mathsf{D}(\kt,\xi')\FT_\grpprime\uf](\kt,\xi',\cdot)$, we conclude $\supp\paropp\uf\subset\halfgrpclosed$.
 
 To show the same property for $\paroppinv\uf$, we employ the version of the Paley-Wiener Theorem presented in Proposition \ref{paley_wiener_prop}. Since $\uf\in \SR(\grp)\subset \LR{2}(\grp)$, we immediately obtain that for fixed $\kt\in\perf\Z$ and $\xi'\in\R^{n-1}$, the Fourier transform $[\FT_\grp\uf](\kt,\xi',\cdot)$ is in the Hardy space $\Hardy_+^2(\R)$. Let
 \begin{align*}
  \widetilde{\Fparopper_+^{-1}}(\kt,\xi',\cdot):\CNumbers_+\to\CNumbers,\quad  \widetilde{\Fparopper_+^{-1}}(\kt,\xi',z):=\prod_{j=1}^m(z-\rho_j^+(\kt,\xi'))^{-1}
 \end{align*}
denote the extension of $\Fparopper_+^{-1}(\kt,\xi',\cdot)$ to the lower complex plane. Since all the roots $\rho_j^+$ lie in the upper complex plane, 
this extension is holomorphic and bounded.
It follows that $[\Fparopper_+^{-1}\FT_\grp\uf](\kt,\xi',\cdot)\in \Hardy_+^2(\R)$. Hence, taking the inverse Fourier transform, Proposition \ref{paley_wiener_prop} yields $\supp\paroppinv\uf\subset\halfgrpclosed$.
\end{proof}

The above properties of $\parop_\pm$ and $\paropinv_\pm$ lead to a surprisingly simple representation formula, see \eqref{weak_sol_lem_repformula} below, for the solution $\uf$ to the problem $\partial_t\uf+\elop\uf=\f$ in the half-space $\torus\times\halfspace$ with Dirichlet boundary conditions.  
The problem itself can be formulated elegantly as \eqref{weak_sol_1}.

\begin{lem}\label{weak_sol_lem}
Assume  $\elop$ is properly elliptic and satisfies Agmon's condition on the two rays $\e^{i\theta}$ with $\theta=\pm\frac{\pi}{2}$.
Let $p\in\np{1,\infty}$ and $\f\in \HSR{-m}{p}_\bot(\grp)$. 
Let $\charplus$ denote the characteristic function on $\halfgrp$.
Then
\begin{align}\label{weak_sol_lem_repformula}
  \uf:=\paroppinv\charplus\paropminv\f,
\end{align}
 is the unique solution in $\HSRcompl{m}{p}(\grp)$ to
 \begin{align}\label{weak_sol_1}
  \supp\uf\subset \halfgrpclosed, \qquad \tand \qquad \supp\np{\parop\uf-\f}\subset \lhalfgrpclosed.
 \end{align}
 Moreover, there is a constant $c=c(n,p)>0$ such that
 \begin{align}\label{weak_sol_est}
  \norm{\uf}_{m,p,\halfgrp}\leq \Ccn{C} \norm{\f}_{-m,p,\halfgrp}.
 \end{align}
\end{lem}

\begin{proof}
By Lemma \ref{halfop_bdd_lem}, $\paropminv\f\in\LRcompl{p}(\grp)$. Clearly then $\charplus\paropminv\f\in\LRcompl{p}(\grp)$ and trivially
$\supp\np{\charplus\paropminv\f}\subset\halfgrpclosed$. 
Lemma \ref{paley_wiener_lem} now implies 
$\supp\np{\paroppinv\charplus\paropminv\f}\subset\halfgrpclosed$, which concludes the first part of  \eqref{weak_sol_1}.
Since 
$\supp\bp{\np{\charplus-\id}\paropminv\f}\subset\lhalfgrpclosed$, 
Lemma \ref{paley_wiener_lem} implies 
$\supp\bp{\paropm\np{\charplus-\id}\paropminv\f}\subset\lhalfgrpclosed$. However, 
$\parop\uf-\f=\paropm\np{\charplus-\id}\paropminv\f$, whence the second part of \eqref{weak_sol_1} follows.

To show uniqueness, let $\vf\in\HSRcompl{m}{p}$ be a solution to \eqref{weak_sol_1} with $f=0$. Since $\paropp\vf=\paropminv\parop\vf$, Lemma \ref{paley_wiener_lem} yields
$\supp\np{\paropp\vf}\subset\lhalfgrpclosed$. On the other hand, $\supp\vf\subset\halfgrpclosed$ by assumption, whence $\supp\paropp\vf\subset\halfgrpclosed$ by Lemma \ref{paley_wiener_lem}.
Recalling that $\paropp\vf\in\LRcompl{p}(\grp)$ by 
Lemma \ref{halfop_bdd_lem}, we thus deduce $\paropp\vf=0$ and consequently $\vf=0$.
This concludes the assertion of uniqueness.

It remains to show \eqref{weak_sol_est}. 
For $F\in\HSRcompl{-m}{p}(\grp)$ with $\supp\np{F-f}\subset\lhalfgrpclosed$, we can utilize Lemma \ref{halfop_bdd_lem} and Lemma \ref{paley_wiener_lem} as 
above to conclude that $\charplus\paropminv F = \charplus\paropminv f$, which in turn implies $\uf=\paroppinv\charplus\paropminv F$.
Recalling Lemma \ref{wholeop_bdd_lem}, we estimate
\begin{align*}
\norm{\uf}_{m,p,\halfgrp} &\leq \norm{\uf}_{m,p}\\
&= \inf\setc{\norm{\paroppinv\charplus\paropminv F}_{m,p}}{F\in\HSRcompl{-m}{p}(\grp),\ \supp\np{F-f}\subset\lhalfgrpclosed}\\
&\leq \Ccn{C}\, \inf\setc{\norm{F}_{-m,p}}{F\in\HSRcompl{-m}{p}(\grp),\ \supp\np{F-f}\subset\lhalfgrpclosed} \\ 
&=\Ccn{C}\, \norm{f}_{-m,p,\halfgrp}.
\end{align*}
\end{proof}

Finally, we can establish the main theorem in the case of the spatial domain being the half-space $\halfspace$ and boundary operators corresponding to Dirichlet boundary 
conditions. 

\begin{thm}\label{dirichlet_sol_thm}
Assume  $\elop$ is properly elliptic and satisfies Agmon's condition on the two rays $\e^{i\theta}$ with $\theta=\pm\frac{\pi}{2}$.
Let $p\in\np{1,\infty}$. For $\f\in \LRcompl{p}(\halfgrp)$ and 
$\g\in\tracespacecompl{\Dtracekonst}{p}(\torus\times\R^n_+)$ 
there is a unique $\uf\in \WSRcompl{1,2m}{p}(\halfgrp)$ subject to
 \begin{align}
 \begin{pdeq}\label{dirichlet_sol_prob}
  \partial_t\uf+\elop \uf&=\f && \tin \halfgrp, \\
  \trace_{m}\uf&=g && \ton \torus\times\partial\R^n_+.
 \end{pdeq}
 \end{align}
 Moreover, there is a constant $c=c(n,p)>0$ such that
 \begin{align}\label{dirichlet_sol_est}
  \norm{\uf}_{\WSR{1,2m}{p}(\halfgrp)}\leq c (\norm{\f}_{\LR{p}(\halfgrp)}+\norm{\g}_{\tracespacecompl{\Dtracekonst}{p}(\torus\times\R^n_+)}).
 \end{align}
\end{thm}

\begin{proof}
We first assume $g=0$. Extending $\f$ by zero to the whole space $\torus\times\R^n$, we have $f\in \LR{p}_\bot(\grp)\subset \HSR{-m}{p}_\bot(\grp)$. 
Let $\uf\in\HSRcompl{m}{p}(\grp)$ be the solution to \eqref{weak_sol_1} from Lemma \ref{weak_sol_lem}. Lemma \ref{BS_TraceSpaceBesselPotentialLem} yields $\trace_{m}\uf=0$. Thus, $\uf$ is a solution to \eqref{dirichlet_sol_prob}. We shall establish
higher order regularity of $\uf$ iteratively. For this purpose, we employ Proposition \ref{TPBesselEqualsTPSobHalfspacecaseProp} to estimate
\begin{align*}
\norm{\uf}_{m+1,p,\halfgrp}
&\leq \Ccn{C} \Bp{\norm{\partial_n^{2m}\uf}_{-m+1,p,\halfgrp} +\sum_{j=0}^{2m-1}\norm{\partial_n^j\op\bb{{\parnorm{k}{\xip}^{2m-j}}}\uf}_{-m+1,p,\halfgrp}}\\
&\leq \Ccn{C} \bp{\norm{\partial_n^{2m}\uf}_{-m+1,p,\halfgrp} +\norm{\op\bb{{\parnorm{k}{\xip}}}\uf}_{m,p,\halfgrp}}.
\end{align*}
Since the symbol of $\parop$ reads $\Fparopper(\kt,\xi',\xi_n)=a\xi_n^{2m}+i\kt+\sum_{k=0}^{2m-1}\sum_{|\alpha|= 2m-k} a_{\alpha,k}(\xi')^\alpha \xi_n^k$ with $a\neq 0$, we deduce with the help of Lemma \ref{MultiplierLem_NegHomoMultipliers} that 
\begin{align*}
\norm{\partial_n^{2m}\uf}_{-m+1,p,\halfgrp}\leq \Ccn{C}\bp{\norm{A\uf}_{-m+1,p,\halfgrp}+\norm{\op\bb{{\parnorm{k}{\xip}}}\uf}_{m,p,\halfgrp}}.
\end{align*}
Consequently,
\begin{align*}
\norm{\uf}_{m+1,p,\halfgrp}
&\leq \Ccn{C} \bp{\norm{f}_{-m+1,p,\halfgrp} +\norm{\op\bb{{\parnorm{k}{\xip}}}\uf}_{m,p,\halfgrp}}.
\end{align*}
Clearly, $\op\bb{\parnorm{k}{\xip}}$ commutes with $\paroppinv\charplus\paropminv$, whence
\begin{align*}
\norm{\uf}_{m+1,p,\halfgrp}
&\leq \Ccn{C} \bp{\norm{f}_{-m+1,p,\halfgrp} +\norm{\paroppinv\charplus\paropminv\op\bb{{\parnorm{k}{\xip}}}\f}_{m,p,\halfgrp}}.
\end{align*}
By Lemma \ref{weak_sol_lem}, we finally obtain
\begin{align*}
\norm{\uf}_{m+1,p,\halfgrp}
&\leq \Ccn{C} \bp{\norm{f}_{-m+1,p,\halfgrp} +\norm{\op\bb{{\parnorm{k}{\xip}}}\f}_{-m,p,\halfgrp}}
\leq \Ccn{C} {\norm{f}_{-m+1,p,\halfgrp}}.
\end{align*}
Iterating this procedure, we obtain after $m$ steps the desired regularity $\uf\in\HSRcompl{2m}{p}(\halfgrp)$ together with the estimate 
$\norm{\uf}_{2m,p,\halfgrp}\leq\Ccn{C}\norm{\f}_p$. Recalling from Proposition \ref{TPBesselEqualsTPSobWholespacecaseProp} that $\HSRcompl{2m}{p}(\halfgrp)=\WSRcompl{1,2m}{p}(\halfgrp)$, we conclude \eqref{dirichlet_sol_est} in the case $g=0$. 
 
If $\g\neq 0$, we recall the properties \eqref{TraceOprMappingproperties} of the trace operator and
choose a function $\vf\in \WSR{1,2m}{p}(\torus\times\R^n_+)$ with $\trace_m\vf=g$ and 
$\norm{\vf}_{2m,p}\leq \Ccn{C} \norm{g}_{\tracespacecompl{\Dtracekonst}{p}(\torus\times\R^n_+)}$. 
 With $\wf:=\uf-\vf$, problem \eqref{dirichlet_sol_prob} is then reduced to
 \begin{align*}
 \begin{pdeq}
  \partial_t\wf+\elop\wf&=\f-(\partial_t\vf+\elop\vf) && \tin \halfgrp, \\
  \trace_{m}\wf&=0 && \ton \torus\times\partial\R^n_+,
 \end{pdeq}
 \end{align*}
 and the assertion readily follows from the homogeneous part already proven. 

To show uniqueness, assume $\uf\in \WSRcompl{1,2m}{p}(\halfgrp)$ is a solution the \eqref{dirichlet_sol_prob} with homogeneous data $f=g=0$.
By Lemma \ref{BS_TraceSpaceBesselPotentialLem} there is an extension $U\in\WSRcompl{1,2m}{p}(\grp)$ of $\uf$ with $\supp U\in\halfgrpclosed$.
By Lemma \ref{weak_sol_lem}, $U=0$.
\end{proof}

\subsection{The Half Space with General Boundary Conditions}\label{HalfSpaceGeneralBoundaryConditionsSection}

Let $\Omega:=\R^n_+$ and $(\elopfull,\bdopfull_1,\ldots,\bdopfull_m)$ be differential operators of the form \eqref{DefOfDiffOprfull} with constant
coefficients. 

\begin{lem}\label{CharmatrixLem}
Assume $\elop$ is properly elliptic 
and $(\elop,\bdop_1,\ldots,\bdop_m)$ satisfies Agmon's complementing condition on the two rays $\e^{i\theta}$ with $\theta=\pm\frac{\pi}{2}$. 
Let $(k,\xi')\in\perf\Z\times\R^{n-1}\minusnullset$. Consider as polynomials in $z$ the mappings $\xin\mapsto\Fparopper_+(k,\xi',\xin)$ and $\xin\mapsto\bdop_j(\xi',\xin)$, where $\Fparopper_+$ is defined as in Lemma \ref{halfop_bdd_lem}.
For $j=1,\ldots,m$ let 
$\charmatrix_{(j-1)l}(k,\xi')$, $l=0,\ldots,m-1$, denote the coefficients of the polynomial
$\bdop_j(\xi',\xin)\mod\Fparopper_+(k,\xi',\xin)$.
The corresponding matrix 
$\charmatrix(k,\xi')\in\C^{m\times m}$ is called \emph{characteristic matrix}. 
The characteristic matrix function 
$\charmatrix:\perf\Z\times\R^{n-1}\minusnullset\ra\C^{m\times m}$ 
has an extension 
$\charmatrix:\R\times\R^{n-1}\minusnullset\ra\C^{m\times m}$ that satisfies $(j,l=0,\ldots,m-1)$
\begin{align}
&\forall\lambda>0:\quad \charmatrix_{jl}(\eta,\xi') = \lambda^{l-m_{j+1}}\charmatrix_{jl}(\lambda^{2m}\eta,\lambda\xi'),\label{CharmatrixLem_Scaling}\\
&\charmatrix_{jl}\in\CRi\bp{\R\times\R^{n-1}\minusnullset}^{m\times m},\label{CharmatrixLem_Smooth}\\
&\snorm{\charmatrix_{jl}(\eta,\xip)}\leq \Ccn[CharmatrixLem_GrowthBoundConst]{C}\,{\parnorm{\eta}{\xip}}^{m_{j+1}-l}.\label{CharmatrixLem_GrowthBound}
\end{align}
Moreover, $\charmatrix(\eta,\xip)$ is invertible and the inverse matrix function $\invcharmatrix(\eta,\xip)$ satisfies 
\begin{align}
&\forall\lambda>0:\quad \invcharmatrix_{jl}(\eta,\xi') = \lambda^{m_{l+1}-j}\invcharmatrix_{jl}(\lambda^{2m}\eta,\lambda\xi'),\label{CharmatrixLem_InvScaling}\\
&\invcharmatrix_{jl}\in\CRi\bp{\R\times\R^{n-1}\minusnullset}^{m\times m},\label{CharmatrixLem_InvSmooth}\\
&\snorm{\invcharmatrix_{jl}(\eta,\xip)}\leq \Ccn[CharmatrixLem_InvGrowthBoundConst]{C}\,{\parnorm{\eta}{\xip}}^{j-m_{l+1}}.\label{CharmatrixLem_InvGrowthBound}
\end{align}
\end{lem}

\begin{proof}
For $j=1,\ldots,m$ the coefficients $\charmatrix_{(j-1)l}(\eta,\xi')$, $l=0,\ldots,m-1$, of the polynomial 
$\xin\mapsto\bb{\bdop_j(\xi',\xin)\mod\Fparop_+(\eta,\xi',\xin)}$ is clearly an extension of the $j$'th row of the characteristic matrix $\charmatrix$ to $(\eta,\xip)\in\R\times\R^{n-1}\minusnullset$. By definition we have
\begin{align}\label{CharmatrixLem_DivWithRest}
\bdop_j(\xi',\xin) = Q_j(\eta,\xi',\xin)\Fparop_+(\eta,\xi',\xin) + \sum_{l=0}^{m-1} \charmatrix_{(j-1)l}(\eta,\xi')\,\xin^l
\end{align}
for some polynomial $\xin\mapsto Q_j(\eta,\xi',\xin)$. Recalling \eqref{FparopDef} and that the roots $\rho_j^+$ are parabolically $1$-homogeneous, we deduce for $\lambda>0$ that 
\begin{align*}
\bdop_j(\xi',\xin) = \lambda^{-m_j}\bdop_j(\lambda \xi',\lambda\xin) = \widetilde{Q_j}(\eta,\xi',\xin)\Fparop_+(\eta,\xi',\xin) + \sum_{l=0}^{m-1} \lambda^{l-m_j}\charmatrix_{(j-1)l}(\lambda^{2m }\eta,\lambda \xi')\,\xin^l
\end{align*}
for some polynomial $\xin\mapsto\widetilde{Q_j}(\eta,\xi',\xin)$.  Comparing coefficients in the two expressions for the polynomial $\xin\mapsto\bdop_j(\xi',\xin)$ above yields \eqref{CharmatrixLem_Scaling}. By \eqref{FparopDef} we have
$\Fparop_+(k,\xi',\xin) = \sum_{\alpha=0}^m \Fparopcoef_\alpha(\eta,\xip)\xin^{m-\alpha}$
with coefficients $\Fparopcoef_\alpha(\eta,\xip)\in\CRi\bp{\R\times\R^{n-1}\minusnullset}$ and $\Fparopcoef_0(\eta,\xip)=1$. 
Polynomial division of $\xin\mapsto\bdop_j(\xip,\xin)$ with $\xin\mapsto\Fparop_+(k,\xi',\xin)$ thus yields a polynomial with coefficients in $\CRi\bp{\R\times\R^{n-1}\minusnullset}$, which establishes \eqref{CharmatrixLem_Smooth}. Choosing $\lambda:=\parnorm{\eta}{\xip}^{-1}$ in \eqref{CharmatrixLem_Scaling} we obtain
$\snorm{\charmatrix_{jl}(\eta,\xip)}\leq {\parnorm{\eta}{\xip}}^{m_{j+1}-l} \sup_{\parnorm{s}{\gamma}=1}\snorm{\charmatrix_{jl}(s,\gamma)}$
and thus \eqref{CharmatrixLem_GrowthBound}. As a direct consequence of Definition \ref{TPADN_AgmonCondAB},  
the rows of $\charmatrix(k,\xi')$ are linearly independent and $\charmatrix(k,\xip)$ therefore invertible for  
$(k,\xip)\in\perf\Z\times\R^{n-1}\minusnullset$.
The scaling property \eqref{CharmatrixLem_Scaling} implies that $\charmatrix(\eta,\xip)$ is also invertible for 
$(\eta,\xip)\in\R\times\R^{n-1}\minusnullset$. The corresponding scaling property \eqref{CharmatrixLem_InvScaling} for the inverse $\invcharmatrix(\eta,\xip)$ follows by multiplying \eqref{CharmatrixLem_Scaling} with 
$\lambda^l\invcharmatrix_{l\alpha}(\eta,\xip)\charmatrix_{\beta j}(\lambda^{2m}\eta,\lambda\xip)$ and summing over $j$ and $l$. Since $\invcharmatrix(\eta,\xip)=\bp{\det F(\eta,\xip)}^{-1}\cof F(\eta,\xip)^\transpose$, \eqref{CharmatrixLem_InvSmooth} follows 
from \eqref{CharmatrixLem_Smooth}. Finally, \eqref{CharmatrixLem_InvGrowthBound} follows from \eqref{CharmatrixLem_InvScaling}
in the same way \eqref{CharmatrixLem_GrowthBound} was derived from \eqref{CharmatrixLem_Scaling}.
\end{proof}

\begin{lem}\label{CharmatrixMultiplierLem}
Let $p\in(1,\infty)$. Under the same assumptions as in Lemma \ref{CharmatrixLem}, the operators $\op\bb{\charmatrix}$ and $\op\bb{\invcharmatrix}$ extend to bounded operators 
\begin{align}
&\op\bb{\charmatrix}:
\tracespacecompl{\Dtracekonst}{p}(\hgrp)\to\tracespacecompl{\Btracekonst}{p}(\hgrp),
\label{CharmatrixMultiplier}\\
&\op\bb{\invcharmatrix}:
\tracespacecompl{\Btracekonst}{p}(\hgrp)\to\tracespacecompl{\Dtracekonst}{p}(\hgrp).\label{InvCharmatrixMultiplier}
\end{align}
\end{lem}

\begin{proof}
For $g\in\SRcompl(\hgrp)^m$ we recall Lemma \ref{BS_SobSloboMultiplierLem} and estimate
\begin{align*}
\norm{\op\bb{\invcharmatrix}g}_{\tracespacecompl{\Dtracekonst}{p}} 
&\leq \sum_{j=1}^m \sum_{l=1}^m
\norm{\iFT_{\hgrp}\bb{\invcharmatrix_{(j-1)(l-1)}\FT_\hgrp\np{g_{l-1}}}}_{\WSRcompl{\Dtracekonst_j,2m\Dtracekonst_j}{p}(\hgrp)}\\
&\leq \sum_{j=1}^m \sum_{l=1}^m
\norm{\iFT_{\hgrp}\bb{\parnorm{k}{\xip}^{m_l-(j-1)}\invcharmatrix_{(j-1)(l-1)}\FT_\hgrp\np{g_{l-1}}}}_{\WSRcompl{\Btracekonst_l,2m\Btracekonst_l}{p}}.
\end{align*}
By \eqref{CharmatrixLem_InvScaling}, $\parnorm{k}{\xi'}^{m_l-(j-1)}\invcharmatrix_{(j-1)(l-1)}(k,\xip)$ is parabolically $0$-homogeneous. 
Corollary \ref{MultiplierLem} thus implies 
\eqref{InvCharmatrixMultiplier}. Assertion \eqref{CharmatrixMultiplier} is shown in a similar manner.
\end{proof}

\begin{lem}\label{HalfspaceGeneralBrdCondsSolutionRepformularLem}
Let the assumptions be as in Lemma \ref{CharmatrixLem}.
Let $(k,\xi')\in\perf\Z\times\R^{n-1}\minusnullset$ and $\gamma$ be a rectifiable Jordan curve in $\C$ that encircles all the roots $\rho_j^+(k,\xi')$ $(j=0,\ldots,m-1)$ of 
$\xin\mapsto\Fparopper_+(k,\xi',\xin)$. Let $c_l^+(k,\xi')\in\C$, $l=0,\ldots,m$, denote coefficients such that 
$\Fparopper_+(k,\xi',\xin)=\sum_{l=0}^m c_l^+(k,\xi')\xin^{m-l}$.
Put
\begin{align*}
\magicvector(k,\xi',\cdot):\R^+\ra\C^m,\quad \magicvector_\alpha(k,\xi',x_n):=
\frac{1}{2\pi i}\int_{\gamma} \frac{\sum_{l=0}^{m-\alpha-1} c_l^+(k,\xi')\xin^{m-\alpha-l}}{\Fparopper_+(k,\xi',\xin)} \e^{ix_n\xin}\,\darg\xin.
\end{align*}
Then for any $g=(g_0,\ldots,g_{m-1})\in\C^m$ the unique solution $u\in\WSR{2m}{2}(\R^+)$ to
\begin{align}\label{HalfspaceDirichletBrdCondsSolutionRepformularLem_ODE}
\begin{pdeq}
& \parop(k,\xi',\partial_n) u = 0 && \tin\R^+,\\
& \trace_{m}u(0) = g && 
\end{pdeq}
\end{align}
is given by $\uf(k,\xi',x_n) =  \magicvector(k,\xi',x_n)\cdot g$.
Moreover,
\begin{align}\label{HalfspaceDirichletBrdCondsSolutionRepformularLem_FormulaDerivatives}
\bdop_j(\xi',\partial_n) \uf(k,\xi',0) = {\charmatrix_{(j-1)l}(k,\xi')g_l}\qquad (j=1,\ldots,m).
\end{align}
\end{lem}

\begin{proof}
With $\uf$ as above, Cauchy's integral theorem yields
\begin{align*}
\parop(k,\xi',\partial_n) u = 
\Bp{\frac{1}{2\pi i}\int_{\gamma} {\sum_{l=0}^{m-\alpha-1} c_l^+(k,\xi')\xin^{m-\alpha-l}}\Fparopper_-(k,\xi',\xin) \e^{ix_n\xin}\,\darg\xin}
g_\alpha
= 0.
\end{align*}
On the same token
\begin{align*}
\bdop_j(\xi',\partial_n) u(0) 
&=  \sum_{\alpha,\beta=0}^{m-1}
\Bp{\frac{1}{2\pi i}\int_{\gamma}\frac{\sum_{l=0}^{m-\alpha-1} c_l^+(k,\xi')\xin^{m-\alpha-l+\beta}}{\sum_{l=0}^m c_l^+(k,\xi')\xin^{m-l}}\,\darg\xin}
\charmatrix_{(j-1)\beta}(k,\xi')g_\alpha.
\end{align*}
Since 
\begin{align*}
\Bp{\frac{1}{2\pi i}\int_{\gamma}\frac{\sum_{l=0}^{m-\alpha-1} c_l^+(k,\xi')\xin^{m-\alpha-l+\beta}}{\sum_{l=0}^m c_l^+(k,\xi')\xin^{m-l}}\,\darg\xin} = \delta_{\alpha\beta},
\end{align*}
which follows by choosing $\gamma$ to be a circle of sufficiently large radius $R$ and letting $R\ra\infty$ (see for example \cite[Chapter 2, Proposition 4.1]{LionsMagenes1}), \eqref{HalfspaceDirichletBrdCondsSolutionRepformularLem_FormulaDerivatives} follows.
\end{proof}

\begin{lem}\label{HalfspaceCharmatrixTransformationBrdValuesL2Setting}
Let the assumptions be as in Lemma \ref{CharmatrixLem}.
If $\uf\in\WSRcompl{1,2m}{2}(\halfgrp)$ satisfies $\parop\uf=0$, then $\bdop\uf= \op\nb{\charmatrix}\trace_m\uf$.
\end{lem}

\begin{proof}
Employ the partial Fourier transform $\FT_{\hgrp}$ to the equation $\parop\uf=0$, which in view of Plancherel's theorem implies
$\parop(k,\xi',\partial_n)\FT_{\hgrp}(\uf)=0$ for almost every $(k,\xi')$. 
By Lemma \ref{HalfspaceGeneralBrdCondsSolutionRepformularLem}, 
$\bdop_j(\xi',\partial_n)\FT_{\hgrp}(\uf)={\charmatrix_{(j-1)l}(k,\xi')\bp{\trace_m u(0)}_l}$. Employing $\iFT_{\hgrp}$, we obtain 
$\bdop\uf= \op\nb{\charmatrix}\trace_m\uf$.
\end{proof}

\begin{thm}\label{generalbrdvalues_sol_thm}
Assume $\elop$ is properly elliptic 
and $(\elop,\bdop_1,\ldots,\bdop_m)$ satisfies Agmon's complementing condition on the two rays $\theta=\pm\frac{\pi}{2}$. 
Let $p\in\np{1,\infty}$. For $\f\in \LR{p}_\bot(\halfgrp)$ and $\g\in\tracespacecompl{\Btracekonst}{p}(\torus\times\partial\R^n_+)$ there is a unique $\uf\in \WSRcompl{1,2m}{p}(\halfgrp)$ subject to
\begin{align}
\begin{pdeq}\label{generalbrdvalues_sol_prob}
 \partial_t\uf+\elop \uf&=\f && \tin \halfgrp, \\
 \bdop\uf&=g && \ton \torus\times\partial\R^n_+.
\end{pdeq}
\end{align}
Moreover, 
\begin{align}\label{generalbrdvalues_sol_est}
 \norm{\uf}_{\WSR{1,2m}{p}(\halfgrp)}\leq c (\norm{\f}_{\LR{p}(\halfgrp)}+\norm{\g}_{\tracespacecompl{\Btracekonst}{p}(\torus\times\partial\R^n_+)}),
\end{align}
where $c=c(n,p)>0$.
\end{thm}

\begin{proof}
As in the proof of Theorem \ref{dirichlet_sol_thm}, it suffices to show existence of a solution to \eqref{generalbrdvalues_sol_prob} satisfying
\eqref{generalbrdvalues_sol_est} for $f=0$ and $g\in\SRcompl(\hgrp)^m$. 
Since $\invcharmatrix$ is smooth away from the origin \eqref{CharmatrixLem_InvSmooth} and has at most polynomial growth  \eqref{CharmatrixLem_GrowthBound},
it follows that $\op\nb{\invcharmatrix} g\in\SRcompl(\hgrp)$. 
Consequently, Theorem \ref{dirichlet_sol_thm} yields existence of a solution 
$\uf\in\WSRcompl{1,2m}{p}(\halfgrp)\cap \WSRcompl{1,2m}{2}(\halfgrp)$ to 
\begin{align}\label{generalbrdvalues_sol_thm_transformedeq}
\begin{pdeq}
 \partial_t\uf+\elop \uf&=0 && \tin \halfgrp, \\
 \trace_m\uf&= \op\nb{\invcharmatrix} g && \ton \torus\times\partial\R^n_+.
\end{pdeq}
\end{align}
From Lemma \ref{HalfspaceCharmatrixTransformationBrdValuesL2Setting} it follows that $\uf$ is in fact a solution to \eqref{generalbrdvalues_sol_prob}.
Additionally, Theorem \ref{dirichlet_sol_thm} and Lemma \ref{CharmatrixMultiplierLem} imply
\begin{align*}
\norm{\uf}_{\WSR{1,2m}{p}(\halfgrp)} 
\leq \Ccn{c}\,\norm{\op\nb{\invcharmatrix} g}_{\tracespacecompl{\Dtracekonst}{p}(\torus\times\partial\R^n_+)}
\leq \Ccn{c}\,\norm{ g}_{\tracespacecompl{\Btracekonst}{p}(\torus\times\partial\R^n_+)}.
\end{align*} 
It remains to show uniqueness. Assume for this purpose that $\uf\in \WSRcompl{1,2m}{p}(\halfgrp)$ is a solution to \eqref{generalbrdvalues_sol_prob} 
with homogeneous right-hand side $f=g=0$. Let $\seqN{g_n}\subset\SRcompl(\hgrp)^m$ be a sequence with $\lim_{n\ra\infty}g_n=\trace_m\uf$ in $\tracespacecompl{\Dtracekonst}{p}(\hgrp)$. By virtue of Theorem \ref{dirichlet_sol_thm} there is a $\uf_n\in\WSRcompl{1,2m}{p}(\halfgrp)\cap \WSRcompl{1,2m}{2}(\halfgrp)$ with $\bb{\partial_t+\elop}\uf_n=0$ and $\trace_m\uf_n=g_n$.
Theorem \ref{dirichlet_sol_thm} and Lemma \ref{CharmatrixMultiplierLem} imply that $\lim_{n\ra\infty}\uf_n=\uf$ in $\WSRcompl{1,2m}{p}(\halfgrp)$ and 
thus $\bdop\uf_n\ra\bdop\uf=0$ in $\tracespacecompl{\Btracekonst}{p}(\hgrp)$. 
By Lemma \ref{HalfspaceCharmatrixTransformationBrdValuesL2Setting}, $\bdop\uf_n=\op\nb{\charmatrix}g_n$. Lemma \ref{CharmatrixMultiplierLem} thus
yields $\trace_m\uf = \lim_{n\ra\infty}g_n=0$. We conclude $\uf=0$ by Theorem \ref{dirichlet_sol_thm}.
\end{proof}

\section{Proof of the Main Theorems}

\begin{proof}[Proof of Theorem \ref{MainThm_HalfSpace}]
As already noted in Section \ref{gr}, the canonical bijection 
between $\CRci\np{\torus\times\R^n_+}$ and $\CRciper\np{\R\times\R^n_+}$ implies that 
$\WSR{1,2m}{p}\bp{\torus\times\R^n_+}$ and $\WSRper{1,2m}{p}(\R\times\R^n_+)$ are isometrically isomorphic.
It follows that 
$\WSRcompl{1,2m}{p}\bp{\torus\times\R^n_+}$ and $\projcompl\WSRper{1,2m}{p}(\R\times\R^n_+)$
as well as $\tracespacecompl{\Btracekonst}{p}(\torus\times\partial\R^n_+)$ and 
$\Pi_{j=1}^m \projcompl\WSRper{\kappa_j,2m\kappa_j}{p}(\R\times\partial\halfspace)$ are 
isometrically isomorphic.
Consequently, Theorem \ref{MainThm_HalfSpace} follows from Theorem \ref{generalbrdvalues_sol_thm}. 
\end{proof}

\begin{proof}[Proof of Theorem \ref{MainThm_TPADN}]
Theorem \ref{MainThm_TPADN} follows from Theorem \ref{MainThm_HalfSpace} by a
standard localization and perturbation argument. One can even apply the argument used 
in the elliptic case \cite{ADN1}; see also \cite[Chapter 4.8]{TanabeBook}.
\end{proof}

\bibliographystyle{abbrv}

\end{document}